\documentclass{article}

\usepackage{amsmath, amssymb, amsthm}
\usepackage[all]{xypic}

\theoremstyle{plain}
\newtheorem{theorem}{Theorem}[section]
\newtheorem{corollary}[theorem]{Corollary}
\newtheorem{proposition}[theorem]{Proposition}
\newtheorem{lemma}[theorem]{Lemma}

\theoremstyle{definition}
\newtheorem{definition}[theorem]{Definition}
\newtheorem{remark}[theorem]{Remark}
\newtheorem{example}[theorem]{Example}

\DeclareMathOperator{\ind}{index}
\DeclareMathOperator{\Kind}{\textit{K}-index}
\DeclareMathOperator{\Gind}{\textit{G}-index}

\DeclareMathOperator{\Ad}{Ad}
\DeclareMathOperator{\ad}{ad}
\DeclareMathOperator{\rank}{rank}
\DeclareMathOperator{\reg}{reg}
\DeclareMathOperator{\ncw}{ncw}

\DeclareMathOperator{\Spin}{Spin}
\DeclareMathOperator{\SO}{SO}
\DeclareMathOperator{\U}{U}
\DeclareMathOperator{\GL}{GL}

\DeclareMathOperator{\Hom}{Hom}
\DeclareMathOperator{\End}{End}

\DeclareMathOperator{\Res}{Res}
\DeclareMathOperator{\DInd}{D-Ind}
\DeclareMathOperator{\KInd}{K-Ind}
\DeclareMathOperator{\HInd}{H-Ind}
\DeclareMathOperator{\HCross}{H-Cross}

\DeclareMathOperator{\HP}{SEHamP}
\DeclareMathOperator{\CHPS}{CSEHamPS}

\DeclareMathOperator{\dom}{dom}

\begin{document}

\title{Quantisation commutes with reduction at discrete series representations of semisimple groups}
\author{Peter Hochs}
\maketitle

\hyphenation{semi-simple pre-quan-ti-sa-tion}

\newcommand{\D}{\slash \!\!\!\! D}

\newcommand{\g}{\mathfrak{g}}
\newcommand{\h}{\mathfrak{h}}
\newcommand{\kk}{\mathfrak{k}}
\newcommand{\torus}{\mathfrak{t}}
\newcommand{\p}{\mathfrak{p}}
\newcommand{\gse}{\mathfrak{g}^*_{\mathrm{se}}}
\newcommand{\kse}{\mathfrak{k}^*_{\mathrm{se}}}

\newcommand{\C}{\mathbb{C}}
\newcommand{\R}{\mathbb{R}}
\newcommand{\Z}{\mathbb{Z}}
\newcommand{\N}{\mathbb{N}}

\newcommand{\B}{\mathcal{B}}
\newcommand{\K}{\mathcal{K}}
\renewcommand{\O}{\mathcal{O}}
\newcommand{\E}{\mathcal{E}}
\newcommand{\SSS}{\mathcal{S}}
\newcommand{\X}{\mathfrak{X}}
\newcommand{\HH}{\mathcal{H}}

\newcommand{\KK}{K \! K}
\newcommand{\Bigwedge}{\textstyle{\bigwedge}}

\newcommand{\QSpinc}{Q_{\Spin^c}}

\newcommand{\ddt}{\left. \frac{d}{dt}\right|_{t=0}}

\begin{abstract}
Using the analytic assembly map that appears in the Baum-Connes conjecture in noncommutative geometry,
 we generalise the Guillemin--Sternberg conjecture that `quantisation commutes with reduction' to (discrete series representations of) semisimple groups $G$ with maximal compact subgroups $K$ acting cocompactly
 on symplectic manifolds. We prove this generalised statement in cases where the image of the momentum map in question lies in the set of strongly elliptic elements
 $\g^*_{\mathrm{se}}$, the set of elements of $\g^*$ with compact stabilisers. This assumption on the image of the momentum map is equivalent to the assumption that
  $M = G \times_K N$, for a compact Hamiltonian $K$-manifold $N$. The proof comes down to a reduction to the compact case. This reduction is based on a   `quantisation commutes with induction'-principle, and involves a notion of induction of Hamiltonian group actions. This principle, in turn, is based on
a version of the naturality of the assembly map for the inclusion  $K \hookrightarrow G$.
\end{abstract}

\tableofcontents

\section*{Introduction}

In this paper we generalise Guillemin and Sternberg's `quantisation commutes with reduction' conjecture to cocompact Hamiltonian actions by semisimple Lie groups. The compact case of this conjecture was proved in \cite{JK,Meinrenken,MS,Paradan1,Paradan2,TZ,Vergne,Ver}. A version for Hamiltonian Lie \emph{groupoid} actions was proved by Bos in \cite{Bos}.

The version of the Guillemin--Sternberg conjecture that we will generalise is the one proved by Paradan in
\cite{Paradan2}. In this version, one considers $\Spin^c$-quantisation, instead of the Dolbeault-quantisation used in \cite{JK,Meinrenken,MS,Paradan1,TZ,Vergne,Ver}. Paradan's result is the following. Suppose a compact Lie group $K$ acts in Hamiltonian fashion on a compact symplectic manifold $(M, \omega)$. Suppose that the cohomology class $[\omega] + \frac{1}{2}c_1(TM, J)$ is integral for some $K$-equivariant almost complex structure $J$ on $M$. If the stabilisers of the action of $K$ on $M$ are abelian, then one has
\begin{equation}\label{eq cpt GS}
\Kind \,  \D_M = \bigoplus_{\lambda \in \Lambda_+} \ind \,  \D_{M_{\lambda + \rho}} V_{\lambda}.
\end{equation}
Here $\D_M$ is a $\Spin^c$-Dirac operator on $M$. Its $K$-index is interpreted as the quantisation of the action of $K$ on $M$. On the right-hand side of \eqref{eq cpt GS}, $\Lambda_+ \subset i\torus^*$ denotes the set of dominant integral weights relative to a choice of maximal torus and positive roots, and $\rho$ is half the sum of the positive roots. By $V_{\lambda}$ we mean the irreducible representation of $K$ with highest weight $\lambda$, and $(M_{\lambda + \rho}, \omega_{\lambda + \rho})$ is the symplectic reduction of $(M, \omega)$ at $-i(\lambda + \rho)$. If this symplectic reduction is not an orbifold (which can occur if $-i(\lambda + \rho)$ is not a regular value of the momentum map), then the index of $\D_{M_{\lambda + \rho}}$ should be replaced by a more subtle definition of the quantisation of $(M_{\lambda + \rho}, \omega_{\lambda + \rho})$.

In this paper, we generalise \eqref{eq cpt GS} to discrete series representations of semisimple Lie groups (Theorem \ref{GSss}). Because we only look at discrete series representations, it is a natural assumption that the image of the momentum map lies inside the set of strongly elliptic elements $\gse \subset \g^*$, defined in \eqref{eq def gse}. Indeed, (some) coadjoint orbits in $\gse$ correspond to discrete series representations, and the quantisation of a Hamiltonian action should decompose into irreducible representations corresponding to coadjoint orbits in the image of the momentum map.

\subsection*{Outline of this paper}

The strategy of our proof of \eqref{eq cpt GS} for a cocompact Hamiltonian action of a semisimple Lie group $G$ on a symplectic manifold $(M, \omega)$ is to reduce this statement to the (known) case of the action of a maximal compact subgroup $K<G$ on the compact submanifold $N:= \bigl(\Phi^M \bigr)^{-1}(\kk^*)$ of $M$, with $\Phi^M: M \to \g^*$ the momentum map. We will see in Section \ref{cpt <-> ncpt} that there are inverse constructions
\[
\begin{split}
\HCross^G_K: \quad & G \circlearrowright M \quad \rightsquigarrow \quad K \circlearrowright N := \bigl(\Phi^M \bigr)^{-1}(\kk^*); \\
\HInd^G_K: \quad & K \circlearrowright N \quad \rightsquigarrow \quad G \circlearrowright M := G \times_K N.
\end{split}
\]
These are called Hamiltonian cross-section and Hamiltonian induction, respectively. In Section \ref{sec preqSpin}, we define induction procedures for prequantisations and $\Spin^c$-structures, compatible with this Hamiltonian induction procedure.

The central result in this paper is Theorem \ref{thm [Q,I]=0}, which states that `quantisation commutes with induction'. Roughly speaking, this is expressed by the diagram
\[
\xymatrix{
(M = G \times_K N, \omega) \ar@{|->}[r]^-{Q_G} & Q_G(M, \omega)   \in K_0(C^*_rG) \\
(N, \nu) \ar@{|->}[u]^{\HInd_K^G} \ar@{|->}[r]^-{Q_K} & Q_K(N, \nu)  \in R(K). \ar@<-5ex>[u]_{\DInd_K^G}
}
\]
Here $R(K)$ is the representation ring of $K$, $K_0(C^*_rG)$ is the $K$-theory of the reduced $C^*$-algebra of $G$, and $\DInd_K^G$ is the Dirac induction map used in the Connes-Kasparov conjecture (see \cite{Wassermann}). In Section \ref{[Q,I]=0}, we tie the other sections in this paper together, by showing how Theorem \ref{thm [Q,I]=0} implies our quantisation commutes with reduction result, Theorem \ref{GSss}, and by sketching a proof of Theorem \ref{thm [Q,I]=0}. The details of this proof are filled in in Sections \ref{natincl} and \ref{Dirac}.

In Section \ref{natincl}, we prove a result (Theorem \ref{natKG}) that can be interpreted as `naturality of the assembly map for the inclusion $K \hookrightarrow G$' (cf.\ \cite{MV}). In Section \ref{Dirac}, we show that this naturality result is well-behaved with respect to the $K$-homology classes of the Dirac operators we use, thus proving Theorem \ref{thm [Q,I]=0}.

\subsection*{Acknowledgements}
The author is very grateful to Paul-\'Emile Paradan, for explaining his ideas in \cite{Paradan1}, for correcting a mistake in a preliminary version of this paper, and for pointing out Lafforgue's article \cite{Lafforgue}, among other
things.  He is also much indebted to Gert Heckman, for advice on group theory and symplectic geometry.

In addition, the author would like to thank Rogier Bos, Eli Hawkins and Klaas Landsman for useful discussions,
and Ulrich Bunke, Siegfried Echterhoff, Dan Kucerovsky and Vincent Lafforgue for kindly answering some questions via e-mail.

\section{Quantisation commutes with reduction for semisimple groups}

In this section, we introduce the terminology we need to state our main result, Theorem \ref{GSss}.

\subsection{$\Spin^c$-quantisation} \label{sec quant Spinc}

Let $(M, \omega)$ be a compact symplectic manifold, equipped with an action by
a compact Lie group $K$, which leaves $\omega$ invariant. Let $J$ be a $K$-equivariant almost complex structure on $M$ (which need not be compatible with $\omega$). Consider the $K$-equivariant line bundle
\[
 \Bigwedge_{\C}^{0, d_M}(TM, J) \to M,
\]
where $d_M$ is the dimension of $M$. Let $L^{2\omega} \to M$ be a $K$-equivariant line bundle whose first Chern class is $[2\omega]$. Suppose that the line bundle
\[
 L^{2\omega} \otimes \Bigwedge_{\C}^{0, d_M}(TM, J) \to M
\]
has a square root $L_J$. Let $P \to M$ be the $\Spin^c$-structure\footnote{We sloppily use the term `$\Spin^c$-structure' for a principal $\Spin^c$-bundle that induces such a structure.} on $M$ associated to $J$ and $L_J$, as described for example in \cite{GGK}, Proposition D.50. Then the determinant line bundle of $P$ is isomorphic to $L^{2\omega}$.

Let $\Delta_{d_M}$ be the standard $2^{d_M/2}$-dimensional representation of $\Spin^c(d_M)$ (see e.g.\ \cite{Friedrich,GVF,LM} ). Let
\[
 \SSS := P \times_{\Spin^c(d_M)} \Delta_{d_M} \to M
\]
be the spinor bundle associated to $P$. The \emph{Clifford action} $c_{TM}$ of $TM$ on $\SSS$ is defined by
\[
 c_{TM}([p,x])[p, \delta] := [p, x\cdot \delta],
\]
where $[p,x] \in P \times_{\Spin^c(d_M)} \R^{d_M} \cong TM$ and $[p,\delta] \in P \times_{\Spin^c(d_M)} \Delta_{d_M} = \SSS$. Here the dot in $x \cdot \delta$ denotes the standard Clifford action of $\R^{d_M}$ on $\Delta_{d_M}$.

Let $\nabla$ be a $K$-equivariant connection on $\SSS$ The \emph{$\Spin^c$-Dirac operator} $\D_M^{L^{2\omega}}$ on $\SSS$, associated to $\nabla$, is defined by the property that for all orthonormal local frames $\{e_1, \ldots, e_{d_M}\}$ of $TM$, one locally has
\[
{\D}_M^{L^{2\omega}} = \sum_{j=1}^{d_M} c_{TM}(e_j) \nabla_{e_j}.
\]

The principal symbol $\sigma_{\D_M^{L^{2\omega}}}$ of $\D_M^{L^{2\omega}}$ is given by
\[
 \sigma_{\D_M^{L^{2\omega}}}(\xi)s = c_{TM}(i\xi^*)s,
\]
where $\xi \in T^*M$, $s \in \SSS$, and $\xi^* \in TM$ is the tangent vector associated to $\xi$ by the Riemannian metric on $M$ induced by the Euclidean metric on $\R^{d_M}$ via the isomorphism $TM
\cong P \times_{\Spin^c(d_M)} \R^{d_M}$. Since
\[
 \sigma_{\D_M^{L^{2\omega}}}(\xi)^2s = -\|\xi\|^2 s
\]
for all $\xi$ and $s$, the $\Spin^c$-Dirac operator $\D_M^{L^{2\omega}}$ is elliptic.

The representation $\Delta_{d_M}$ of $\Spin^c(d_M)$ has a natural $\Z_2$-grading $\Delta_{d_M} = \Delta_{d_M}^+ \oplus \Delta_{d_M}^-$, which induces a grading
\[
 \SSS = \SSS^+ \oplus \SSS^-.
\]
Since the Clifford action of $\R^{d_M}$ on $\Delta_{d_M}$ interchanges the subspaces $\Delta_{d_M}^{\pm}$, we have two operators
\[
 \begin{split}
  \D_M^+ := \D_M^{L^{2\omega}}|_{\Gamma^{\infty}(\SSS^+)}: &\Gamma^{\infty}(\SSS^+) \to \Gamma^{\infty}(\SSS^-); \\
 \D_M^- := \D_M^{L^{2\omega}}|_{\Gamma^{\infty}(\SSS^-)}: &\Gamma^{\infty}(\SSS^-) \to \Gamma^{\infty}(\SSS^+).
 \end{split}
\]
Because $\D_M$ is elliptic and $M$ is compact, the kernels of the operators $\D_M^{\pm}$ are finite-dimensional representations of $K$. Since $\D_M^{L^{2\omega}}$ is symmetric with respect to the $L^2$-inner product on compactly supported sections of $\SSS$ (see e.g.\ \cite{GVF}, Proposition 9.13), the operators $\D_M^{\pm}$ are each other's formal adjoints. We slightly abuse terminology by setting
\[
 \Kind \, \D_M^{L^{2\omega}} := [\ker \D_M^+] - [\ker \D_M^-] \quad \in R(K),
\]
the representation ring of $K$. This index is by definition the \emph{$\Spin^c$-quantisation of the action of $K$ on $(M, \omega)$}:
\begin{definition}
 \[
  \QSpinc^K(M, \omega) := \Kind \, \D_M^{L^{2\omega}} \quad \in R(K).
 \]
\end{definition}

\subsection{Quantisation commutes with reduction, the compact case}

We will continue to use the notation and assumptions of Subsection \ref{sec quant Spinc}. Now suppose in addition that the action of $K$ on $(M, \omega)$ is \emph{Hamiltonian}, and let $\Phi: M \to \kk^*$ be a momentum map.

\subsubsection*{Quantisation of symplectic reductions}

Suppose that $\xi \in \kk^*$ is a regular value of $\Phi$, and that the stabiliser $K_{\xi}$ acts freely on $\Phi^{-1}(\xi)$. Then the symplectic reduction $M_{\xi} := \Phi^{-1}(\xi)/{K_{\xi}}$ is a smooth manifold.

In \cite{Paradan2}, Paradan shows that $P$ induces a $\Spin^c$-structure $P_{\xi}$ on $M_{\xi}$ whose determinant line bundle is $L^{2\omega_{\xi}}$. The $\Spin^c$-quantisation of $(M_\xi, \omega_{\xi})$ is then defined, as in Subsection \ref{sec quant Spinc}, as the index of the $\Spin^c$-Dirac operator $\D_{M_{\xi}}^{L^{2\omega_{\xi}}}$ on the spinor bundle $\SSS_{\xi}$ of $P_{\xi}$, with respect to any connection on $\SSS_{\xi}$:
\[
\QSpinc(M_{\xi}, \omega_{\xi}) = \ind \, \D_{M_{\xi}}^{L^{2\omega_{\xi}}}.
\]

Even if the action of $K_{\xi}$ on $\Phi^{-1}(\xi)$ is not assumed to be free, it is still locally free by Smale's lemma. Then the reduced space $M_{\xi}$ is an orbifold. It is then still possible to define a  $\Spin^c$-Dirac operator on $M_{\xi}$, and its index is still denoted by $\QSpinc(M_{\xi}, \omega_{\xi})$. This index can be computed via Kawasaki's orbifold index theorem (see \cite{kawasaki}, or \cite{Meinrenken}, Theorem 3.3).

\subsubsection*{Quantisation commutes with reduction}

Let $T < K$ be a maximal torus, with Lie algebra $\torus \subset \kk$. Let $\torus^*_+ \subset \torus^*$ be a choice of positive Weyl chamber. Let $R^+$ be the set of positive roots of $(\kk, \torus)$ with respect to $\torus^*_+$, and write $\rho := \frac{1}{2}\sum_{\alpha \in R^+} \alpha$.

Let $\Lambda_+ \subset i\torus^*_+$ be the set of dominant weights of $(\kk, \torus)$. For $\lambda \in \Lambda_+$, we will denote the irreducible representation of $K$ with highest weight $\lambda$ by $V_{\lambda}$. Let $R^{\lambda}_K: R(K) \to \Z$ be the multiplicity function of $V_{\lambda}$. We wil write
$(M_{\lambda}, \omega_{\lambda}) := (M_{-i\lambda}, \omega_{-i\lambda})$.

The $\Spin^c$-version of the Guillemin--Sternberg conjecture is the following statement. This is Theorem 1.7 from \cite{Paradan2}.
\begin{theorem} \label{thm [Q,R]=0 cpt spin}
 If the stabilisers of the action of $K$ on $M$ are \emph{abelian}, the for all $\lambda \in \Lambda_+ \cap i\Phi(M)$,
\[
 R_K^{\lambda}\bigl(\QSpinc(M, \omega) \bigr) = \QSpinc(M_{\lambda + \rho}, \omega_{\lambda + \rho}).
\]
If $\lambda \in \Lambda_+ \setminus i\Phi(M)$, then the integer on the left hand side equals zero.
\end{theorem}
The condition that the action of $K$ on $M$ has abelian stabilisers is related to the fact that there may be several different coadjoint orbits in $\kk^*$ whose $\Spin^c$-quantisation equals a given irreducible representation of $K$. This ambiguity, which is not present in the case of Dolbeault-quantisation, can be removed by imposing the condition that the action has abelian stabilisers.

\subsection{Noncompact groups and manifolds}

Now suppose that $G$ is any Lie group, acting on a possibly noncompact symplectic manifold $(M, \omega)$, leaving $\omega$ invariant. The $\Spin^c$-quantisation of this action cannot be defined as in Subsection \ref{sec quant Spinc}, since the kernel of an elliptic operator on a noncompact manifold need not be finite-dimensional. Furthermore, the representation ring of a noncompact group is not well-defined.

Therefore, in \cite{Landsman} and \cite{HL}, it is proposed to define the quantisation of this action using the $K$-theory group $K_0(C^*G)$ of the $C^*$-algebra of $G$ instead of the representation ring,
and the \emph{analytic assembly map}
\[
 \mu_M^G: K_0^G(M) \to K_0(C^*G)
\]
instead of the equivariant index of elliptic operators (see \cite{BCH,Valette}). Here $K_0^G(M)$ is the equivariant $K$-homology group of $M$ (see \cite{HR}). The Dirac operator $\D_M^{L^{2\omega}}$ defines an element $\bigl[\D_M^{L^{2\omega}}\bigr]$ of this group, if the orbit space $M/G$ is \emph{compact}. We will assume compactness of $M/G$ throughout this paper. If both $M$ and $G$ are compact, then $K_0(C^*G) \cong R(G)$, and this isomorphism maps $\mu_M^G\bigl[\D_M^{L^{2\omega}}\bigr]$ to $\Gind \, \D_M^{L^{2\omega}}$.

In this paper, we will use the reduced $C^*$-algebra $C^*_rG$ of $G$ instead of the full one used in \cite{HL,Landsman}. In those papers, one considers reduction at the trivial representation, which is not tempered. Therefore, the reduction map used in \cite{HL,Landsman} is not well-defined on $K_0(C^*_rG)$. In this paper, we will consider discrete series representations of semisimple Lie groups. It follows from the fact that these representations are tempered, that the reduction map defined in Subsection \ref{Laff} \emph{is} well-defined on $K_0(C^*_rG)$.

With these replacements, we get
\begin{definition} \label{def quant VI}
 \[
    \QSpinc^G(M, \omega) := \mu_M^G\bigl[\D_M^{L^{2\omega}} \bigr] \quad \in K_0(C^*_rG).
 \]
\end{definition}

In the papers \cite{HL,Landsman}, a reduction map
\[
 R_G^0: K_0(C^*G) \to \Z
\]
is defined, which is used to state a `quantisation commutes with reduction'-conjecture. In these papers, as in most of the literature on the Guillemin--Sternberg conjecture, one does not use the $\Spin^c$-Dirac operator defined above, but
the Dolbeault--Dirac operator. Or  equivalently,  a $\Spin^c$-Dirac operator that acts on sections of the same vector bundle as the Dolbeault--Dirac opertor, and has the same principal symbol. For Lie groups $G$ with a normal discrete subgroup $\Gamma \lhd G$ such that $G/\Gamma$ is compact, we prove this generalised Guillemin--Sternberg conjecture in \cite{HL}.

In this paper, we will define reduction maps at discrete series representations of semisimple Lie groups, and deduce a `quantisation commutes with reduction'-result (Theorem \ref{GSss}) from the compact case, Theorem \ref{thm [Q,R]=0 cpt spin}. These reduction maps were first defined in V.\ Lafforgue's version of Atiyah \& Schid's \cite{ASch} and Parthasarathy's \cite{Parthasarathy} work, as explained in Subsection \ref{Laff}.

\subsection{Discrete series representations and $K$-theory} \label{Laff}

In \cite{Lafforgue}, V.\ Lafforgue reproves some classical results about discrete series representations by Harish-Chandra \cite{HC1,HC2},
Atiyah \& Schmid \cite{ASch} and Parthasarathy \cite{Parthasarathy}, using $K$-homology, $K$-theory and assembly maps. We will give a quick summary of the results in
\cite{Lafforgue} that we will use in this thesis.

For the remainder of this subsection,
let $G$ be a connected\footnote{Theorem \ref{GSss} and the other results in this paper (possibly in modified forms) are also valid for groups with finitely many connected components, but the assumption
that $G$ is connected allows us to circumvent some technical difficulties.} semisimple Lie group with finite centre. Let $K < G$ be a maximal compact subgroup, and let $T < K$ be a maximal torus.
Suppose that $T$ is also a Cartan subgroup of $G$, so that $G$ has discrete series
representations by Harish-Chandra's criterion \cite{HC2}. Discrete series representations\index{discrete series representation} are representations whose matrix elements are square-integrable over $G$. They form a discrete subset of the unitary dual of $G$.

In \cite{Parthasarathy}, Parthasarathy realises the irreducible discrete series representations of $G$ as the $L^2$-indices of Dirac operators ${\D}^V$,
where $V$
runs over the irreducible representations of $K$. Atiyah and Schmid do the same in \cite{ASch}, replacing Harish-Chandra's work by results from index theory.
In \cite{Slebarsky}, Slebarsky
considers the decomposition into irreducible representations of $G$ of $L^2$-indices of Dirac operators on any homogeneous space $G/L$, with $L<G$ a compact, connected
subgroup.

\subsubsection*{Dirac induction}

For a given irreducible representation $V$ of $K$, the Dirac operator ${\D}^V$ used by Parthasarathy and Atiyah--Schmid is defined as follows.
Let $\p \subset \g$\index{paad@$\p$} be the orthogonal complement to $\kk$ with respect to the Killing form. Then $\p$ is an $\Ad(K)$-invariant linear subspace
of $\g$, and $\g = \kk \oplus \p$. Consider the inner product on $\p$ given by the restriction of the Killing form.
The adjoint representation
\[
\Ad: K \to \GL(\p)
\]
of $K$ on $\p$ takes values in $\SO(\p)$, because the Killing form is $\Ad(K)$-invariant, and $K$ is connected. We suppose that it
has a lift $\widetilde{\Ad}$\index{aa@$\widetilde{\Ad}$} to the double cover $\Spin(\p)$ of $\SO(\p)$.
It may be necessary to replace $G$ and $K$ by double covers for this lift to exist. Then the homogeneous space
$G/K$ has a $G$-equivariant $\Spin$-structure
\[
P^{G/K} := G \times_K \Spin(\p) \to G/K.
\]
Here $G \times_K \Spin(\p)$ is the quotient of $G \times \Spin(\p)$ by the action of $K$ defined by
\[
 k(g,a) = (gk^{-1}, \widetilde{\Ad}(k)a),
\]
for $k \in K$, $g \in G$ and $a \in \Spin(\p)$.

Fix an orthonormal basis $\{X_{1}, \ldots, X_{d_{\p}}\}$ of $\p$. Using this basis, we identify $\Spin(\p) \cong \Spin(d_{\p})$.
Let $\Delta_{d_{\p}}$ be the
canonical $2^{\frac{d_{\p}}{2}}$-dimensional representation of $\Spin(d_{\p})$ (see Subsection \ref{sec quant Spinc}). Because $\p$ is even-dimensional, $\Delta_{d_{\p}}$ splits into two irreducible subrepresentations $\Delta_{d_{\p}}^+$ and $\Delta_{d_{\p}}^-$.
Consider the $G$-vector bundles
\[
E_V^{\pm} := G \times_{K} (\Delta_{d_{\p}}^{\pm} \otimes V) \to G/K.
\]
Note that
\begin{equation} \label{secEV}
\Gamma^{\infty}(G/K, E^{\pm}_V) \cong \bigl(C^{\infty}(G)\otimes \Delta_{d_{\p}}^{\pm} \otimes V  \bigr)^K,
\end{equation}
where $K$ acts on $C^{\infty}(G)\otimes \Delta_{d_{\p}}^{\pm} \otimes V$ by
\begin{equation}\label{Kact}
k\cdot (f \otimes \delta \otimes v) = (f \circ l_{k^{-1}} \otimes \widetilde{\Ad}(k) \delta \otimes k\cdot v)
\end{equation}
for all $k \in K$, $f \in C^{\infty}(G)$, $\delta \in \Delta_{d_{\p}}$ and $v \in V$. Here $l_{k^{-1}}$ denotes left multiplication by $k^{-1}$.

Using the basis $\{X_{1}, \ldots, X_{d_{\p}}\}$ of $\p$ and the isomorphism \eqref{secEV}, define the differential operator
\begin{equation} \label{eq dirac G/K}
{\D}^V: \Gamma^{\infty}(E_V^+) \to \Gamma^{\infty}(E_V^-)
\end{equation}
by the formula
\begin{equation} \label{DV}
{\D}^V := \sum_{j=1}^{d_{\p}} X_j \otimes c(X_j) \otimes 1_V.
\end{equation}
Here in the first factor, $X_j$ is viewed as a left invariant vector field on $G$, and in the second factor, $c: \p \to \End(\Delta_{d_{\p}})$
is the Clifford action (see Subsection \ref{sec quant Spinc}). This action is odd with respect to the grading on $\Delta_{d_{\p}}$.
The operator \eqref{eq dirac G/K} is the $\Spin$-Dirac operator on $G/K$ (see \cite{Parthasarathy}, Proposition 1.1 and \cite{Friedrich}, Chapter 3.5).

Lafforgue (see also Wassermann \cite{Wassermann}) uses the same operator to define a `Dirac induction map'\index{Dirac induction} \index{daaaaaaaaa@$\DInd_K^G$}
\begin{equation} \label{eq DInd}
\DInd_K^G: R(K) \to K_0(C^*_r(G))
\end{equation}
by
\begin{equation}\label{IGK}
\DInd_K^G[V] := \left[ \bigl(C^*_r(G) \otimes \Delta_{d_{\p}} \otimes V \bigr)^K, b \bigl({\D}^V \bigr) \right],
\end{equation}
where $b: \R \to \R$ is a normalising function, e.g.\ $b(x) = \frac{x}{\sqrt{1+x^2}}$. The expression on the right hand side defines a class in Kasparov's $\KK$-group
$\KK_0(\C, C^*_r(G))$, which is isomorphic to the $K$-theory group $K_0(C^*_r(G))$.
In \cite{Wassermann}, Wassermann proves the Connes--Kasparov conjecture, which states that this Dirac induction map is a bijection for linear reductive groups.

\subsubsection*{Reduction}

The relation between the Dirac induction map and the work of Atiyah \& Schmid and of Parthasarathy can be seen by embedding the discrete series of $G$ into $K_0(C^*_r(G))$ via the map
\[
\HH \mapsto [\HH] := [d_{\HH}  c_{\HH}],
\]
where $\HH$ is a Hilbert space with inner product $(\relbar, \relbar )_{\HH}$, equipped with a discrete series representation of $G$, $c_{\HH} \in C(G)$ is the function
\[
c_{\HH}(g) = (\xi, g\cdot \xi)_{\HH}
\]
(for a fixed $\xi  \in \HH$ of norm $1$), and $d_{\HH}$ is the inverse of the $L^2$-norm of $c_{\HH}$ (so that the function $d_{\HH} c_{\HH}$ has $L^2$-norm $1$).
Because $d_{\HH} c_{\HH}$ is a projection in $C^*_r(G)$, it indeed defines a class in
$K_0(C^*_r(G))$.

Next, Lafforgue defines a map\footnote{In Lafforgues's notation, $R_G^{\HH}(x) = \langle \HH, x\rangle$.}
\begin{equation} \label{RG} \index{quantum reduction!at discrete series representations} \index{raag@$R^{\HH}_G$}
R^{\HH}_G: K_0(C^*_r(G)) \to \Z
\end{equation}
that amounts to taking the multiplicity of the irreducible discrete series representation $\HH$, as follows. Consider the map
\[
C^*_r(G) \to \K(\HH)
\]
(the $C^*$-algebra of compact operators on $\HH$), given on $C_c(G) \subset C^*_r(G)$ by
\begin{equation} \label{RGH}
f \mapsto \int_G f(g) \, \pi(g) \, dg.
\end{equation}
Here $\pi$ is the representation of $G$ in $\HH$. Since $K_0(\K(\HH)) \cong \Z$, this map induces a map $K_0(C^*_r(G)) \to \Z$ on $K$-theory, which by definition is \eqref{RG}.

The map $R^{\HH}_G$ has the property that for all irreducible discrete series representations $\HH$ and $\HH'$ of $G$, one has
\[
R_G^{\HH}([\HH']) = \left\{ \begin{array}{cl} 1 & \text{if $\HH \cong \HH'$ } \\ 0 & \text{if $\HH \not\cong \HH'$.}\end{array} \right.
\]
Hence it can indeed be interpreted as a multiplicity function. For compact groups, it follows from Schur orthogonality that this is indeed the usual multiplicity.

\medskip
\noindent
Dirac induction links the reduction map $R_G^{\HH}$ to the usual reduction map defined by taking multiplicietis of a given representation in the following way.

Let $R = R(\g, \torus)$ be the root system\index{raab@$R$, $R_c$, $R_n$} of $(\g, \torus)$, let $R_c := R(\kk, \torus) \subset R$ be the subset of compact roots, and let $R_n := R \setminus R_c$ be the set of noncompact roots. Let $R_c^+ \subset R_c$\index{raac@$R^+$, $R_c^+$, $R_n^+$} be a choice of positive compact roots, and let $\Lambda^{\kk}_+$ be the set of dominant integral weights of $(\kk, \torus)$ with respect to $R_c^+$.

Let $\HH$ be an irreducible discrete series representation of $G$. Let $\lambda$ be the Harish-Chandra parameter\index{Harish-Chandra parameter} of $\HH$ (see \cite{HC1,HC2})
 such that $(\alpha, \lambda) > 0$ for all $\alpha \in R_c^+$. Here $(\relbar, \relbar)$ is a Weyl group invariant inner product on $\torus^*_{\C}$. Let $R^+ \subset R$ be the positive root system defined by
\[
\alpha \in R^+ \quad \Leftrightarrow \quad (\alpha, \lambda) > 0,
\]
for $\alpha \in R$. Then $R_c^+ \subset R^+$, and we denote by $R_n^+ := R^+ \setminus R_c^+$ the set of noncompact positive roots. We will write $\rho:= \frac{1}{2}\sum_{\alpha \in R^+} \alpha$\index{raad@$\rho$, $\rho_c$} and $\rho_c:= \frac{1}{2}\sum_{\alpha \in R_c^+} \alpha$. We will use the fact that $\lambda - \rho_c$ lies on the dominant weight lattice $\Lambda^{\kk}_+$, since $\lambda \in \Lambda^{\kk}_+ + \rho$.

Note
that the dimension of the quotient $G/K$ equals the number of noncompact roots, which is twice the number of positive noncompact roots, and hence even.
\begin{lemma} \label{lem Laff}
Let $\mu \in \Lambda^{\kk}_+$ be given. Let $V_{\mu}$ be the irreducible representation of $K$ with highest weight $\mu$.
 We have
\begin{equation} \label{eq red disc cpt}
R_G^{\HH}\bigl(\DInd_K^G [V_{\mu}]\bigr) = \left\{\begin{array}{cl} (-1)^{\frac{\dim G/K}{2}} & \text{if $\mu = \lambda - \rho_c$} \\
0 & \text{otherwise.} \end{array} \right.
\end{equation}
\end{lemma}
The relation \eqref{eq red disc cpt} can be summarised as
\[
R_G^{\HH} \circ \DInd_K^G =  (-1)^{\frac{\dim G/K}{2}} R_K^{\lambda - \rho_c},
\]
with $R_K^{\lambda - \rho_c}: R(K) \to \Z$ given by taking multiplicities of the irreducible $K$-representation with highest weight $\lambda - \rho_c$.

\begin{proof}
According to Lafforgue \cite{Lafforgue}, Lemma 2.1.1, we have
\begin{align}
R_G^{\HH}\bigl(\DInd_K^G [V_{\mu}]\bigr) &= \dim\bigl(V^*_{\mu} \otimes \Delta_{d_{\p}}^* \otimes \HH \bigr)^K \nonumber \\
	&= \bigl[\Delta_{d_{\p}}^* \otimes \HH|_K : V_{\mu}\bigr], \label{eq mult}
\end{align}
the multiplicity of $V_{\mu}$ in $\Delta_{d_{\p}}^* \otimes \HH|_K$. Let us compute this multiplicity.

By Harish-Chandra's formula (Harish-Chandra \cite{HC2}, Schmid \cite{Schmid}, Theorem on page 95/96),
the character $\Theta_{\lambda}$ of $\HH$ is given by
\[
\Theta_{\lambda}|_{T^{\reg}} = (-1)^{\frac{\dim G/K}{2}}
\frac{\sum_{w \in W(\kk, \torus)} \varepsilon(w)e^{w\lambda}}{\prod_{\alpha \in R^+} \bigl(e^{\alpha/2} - e^{-\alpha/2}\bigr).}
\]
Here $\varepsilon(w) = \det(w)$, and $W(\kk, \torus)$ is the Weyl group of $(\kk, \torus)$.
The character $\chi_{\Delta_{d_{\p}}}$ of the representation
\begin{equation} \label{eq tilde Ad}
K \xrightarrow{\widetilde{\Ad}} \Spin(\p) \to \GL(\Delta_{d_{\p}}),
\end{equation}
on the other hand,
is given by (Parthasarathy \cite{Parthasarathy}, Remark 2.2)
\[
\chi_{\Delta_{d_{\p}}}|_{T^{\reg}} := \bigr(\chi_{\Delta_{d_{\p}}^+} - \chi_{\Delta_{d_{\p}}^-}\bigr)|_{T^{\reg}} = \prod_{\alpha \in R_n^+} \bigl(e^{\alpha/2} - e^{-\alpha/2}\bigr).
\]
It follows from this formula that for all $t \in T^{\reg}$,
\[
\chi_{\Delta_{d_{\p}}^*}(t) = \overline{\chi_{\Delta_{d_{\p}}}(t^{-1})} = \chi_{\Delta_{d_{\p}}}(t),
\]
and hence
\[
\begin{split}
\bigl(\Theta_{\lambda}\chi_{\Delta_{d_{\p}}^*}\bigr)|_{T^{\reg}} &=
(-1)^{\frac{\dim G/K}{2}}
\frac{\sum_{w \in W(\kk, \torus)} \varepsilon(w)e^{w\lambda}}{\prod_{\alpha \in R_c^+} \bigl(e^{\alpha/2} - e^{-\alpha/2}\bigr) } \\
 & = (-1)^{\frac{\dim G/K}{2}} \chi_{\lambda-\rho_c},
\end{split}
\]
by Weyl's character formula. Here $\chi_{\lambda-\rho_c}$ is the character of the irreducible representation of $K$ with highest weight $\lambda-\rho_c$.

Therefore, by \eqref{eq mult},
\[
\begin{split}
R_G^{\HH}\bigl(\DInd_K^G [V_{\mu}]\bigr) 	&= \bigl[\Delta_{d_{\p}}^* \otimes \HH|_K : V_{\mu}\bigr] \\
	&= (-1)^{\frac{\dim G/K}{2}} [V_{\lambda-\rho_c}: V_{\mu}] \\
	&= \left\{\begin{array}{cl} (-1)^{\frac{\dim G/K}{2}} & \text{if $\mu = \lambda - \rho_c$} \\
					0 & \text{otherwise.} \end{array} \right.
\end{split}
\]
\end{proof}

\begin{remark}
Lemma \ref{lem Laff} is strictly speaking not an orbit method, because the coadjoint orbit through $\mu$ is only equal to $G/K$ if $K=T$, and $\mu$ does not lie on any root hyperplanes.
\end{remark}

\subsection[${[Q,R]}=0$ for semisimple groups]{Quantisation commutes with reduction at discrete series representations of semisimple groups} \label{sec GSss}

Consider the situation of Definition \ref{def quant VI}, with the additional assumptions and notation of Subsection \ref{Laff}. Suppose that the action of $G$ on $M$ is Hamiltonian, with momentum map $\Phi$. We will state a generalisation of Theorem \ref{thm [Q,R]=0 cpt spin} in this setting, under the assumption that the image of $\Phi$ lies inside the \emph{strongly elliptic set} $\gse \subset \g^*$. We first clarify this assumption, and then state our result for semisimple groups.

\subsubsection*{The set $\gse$}

Let us define the subset $\gse \subset \g^*$ of \emph{strongly elliptic elements}.\index{strongly elliptic element}
We always view $\kk^*$ as a subspace of $\g^*$ via the linear isomorphism  $\kk^* \cong \p^0$
(via restriction from $\g$ to $\kk$), with $\p^0$ the annihilator of $\p$ in $\g^*$.
As before, the dual space $\torus^*$ is identified with the subspace $\bigl(\kk^*\bigr)^{\Ad^*(T)}$ of $\kk^*$.

Let $\torus^*_+ \subset \torus^*$ be a choice of positive Weyl chamber. We denote by `$\ncw$'\index{n@$\ncw$} the set of noncompact walls:
\begin{equation} \label{eq def ncw}
\ncw := \{\xi \in \torus^*; (\alpha, \xi) =0 \text{ for some $\alpha \in R_n$} \},
\end{equation}
where as before, $(\relbar, \relbar)$ is a Weyl group invariant inner product on $\torus^*_{\C}$. We then define\index{gaa@$\gse$}
\begin{equation} \label{eq def gse}
\gse := \Ad^*(G) (\torus_+^* \setminus \ncw).
\end{equation}
Equivalently, $\gse$ is the set of all elements of $\g^*$ with compact stabilisers under the coadjoint action, and also the interior of the elliptic set
$\g^*_{\mathrm{ell}}:= \Ad(G)\kk^*$. We will also use the notation
\begin{equation} \label{eq def kse}
\kse := \Ad^*(K) (\torus_+^* \setminus \ncw).
\end{equation}
Note that $\kse \subset \kk^*$ is an open dense subset, and that $\gse = \Ad^*(G) \kse$. The set $\gse$ is generally not dense in $\g^*$.

The reason for our assumption that the momentum map takes values in $\gse$ is that we are looking at multiplicities of discrete series representations. These can be seen as `quantisations' of certain coadjoint orbits that lie inside $\gse$ (see Schmid \cite{Schmid}, Parthasarathy \cite{Parthasarathy} and also Paradan \cite{Paradan2}). In general, the `quantisation commutes with reduction' principle implies that the quantisation of a Hamiltonian action decomposes into irreducible representations associated to coadjoint orbits that lie in the image of the momentum map. Hence if we suppose that this image lies inside $\gse$, we expect the quantisation of the action to decompose into discrete series representations.
In \cite{Weinstein}, Proposition 2.6, Weinstein proves that $\gse$ is nonempty if and only if $\rank G = \rank K$, which is Harish-Chandra's criterion for the existence of discrete series representations of $G$.

The most direct application of the assumption that the image of the momentum map lies in $\gse$ is the following lemma, which we will use several times.
\begin{lemma} \label{lem gxi p}
Let $\xi \in \gse$. Then $\g_{\xi} \cap \p = \{0\}$.
\end{lemma}
\begin{proof}
Let $X \in \g_{\xi} \cap \p$ be given. We consider the one-parameter subgroup $\exp (\R X)$ of $G$. Because $\xi \in \gse$, the stabiliser $G_{\xi}$ is compact. Because
$\exp(\R X)$ is contained in $G_{\xi}$, it is therefore either the image of a closed curve, or dense in a subtorus of $G_{\xi}$. In both cases, its closure is compact.

On the other hand, the map $\exp: \p \to G$ is an embedding (see e.g.\ \cite{Knapp}, Theorem 6.31c). Hence, if $X\not=0$, then $\exp(\R X)$ is a closed subset of $G$,
diffeomorphic to $\R$. Because the closure of $\exp(\R X)$ is compact by the preceding argument, we conclude that $X=0$.
\end{proof}

Now suppose that $\Phi(M) \subset \gse$. Then the assumption that the action of $G$ on $M$ is proper is actually unnecessary:
\begin{lemma}
If $\Phi(M) \subset \gse$, then the action of $G$ on $M$ is automatically proper.
\end{lemma}
\begin{proof}
In \cite{Weinstein}, Corollary 2.13, it is shown that the coadjoint action of $G$ on $\gse$ is proper. This is a slightly stronger property than the fact that elements of
$\gse$ have compact stabilisers, and it implies properness of the action of $G$ on $M$.

Indeed, let a compact subset $C \subset M$ be given. It then follows from continuity and equivariance of $\Phi$, and from properness of the action of $G$ on $\gse$ that the closed
set
\[
\begin{split}
G_C &:= \{g \in G; gC \cap C \not= \emptyset\} \\
& \subset \{g \in G; g\Phi(C) \cap \Phi(C) \not= \emptyset\}
\end{split}
\]
is compact, i.e. the action of $G$ on $M$ is proper.
\end{proof}

\subsubsection*{The result}

Compactness of $M/G$ is enough to guarantee compactness of the reduced spaces
$M_{\xi} = \Phi^{-1}(\xi)/G_{\xi} \cong \Phi^{-1}(G \cdot \xi)/G$, but it can even be shown that in this
setting, $\Phi$ is a proper map. This gives another reason why the reduced spaces are compact.

We can finally state our result. Let $\HH$ be an irreducible discrete series representation. Let $\lambda \in i\torus^*$ be its Harish-Chandra parameter such that $(\alpha, \lambda) > 0$ for all $\alpha \in R_c^+$.
As before, we will write $(M_{\lambda}, \omega_{\lambda}) := (M_{-i\lambda}, \omega_{-i\lambda})$ for the symplectic reduction of $(M, \omega)$ at $-i\lambda \in \torus^*_+ \setminus \ncw \subset \gse$.
Then our generalisation of Theorem \ref{thm [Q,R]=0 cpt spin} is:
\begin{theorem}[Quantisation commutes with reduction at discrete series representations] \label{GSss} \index{quantisation commutes with reduction!at discrete series representations}
Consider the situation of Definition \ref{def quant VI}. Suppose that the action of $G$ on $M$ is proper and Hamiltonian, and that the additional assumptions of this subsection hold. Supose furthermore that the action of $G$ on $M$ has \emph{abelian} stabilisers. If $-i\lambda$ is a regular value of $\Phi$, then
\[
\boxed{
R^{\HH}_G\bigl(Q_{V\!I}(M, \omega)\bigr) :=
R^{\HH}_G \bigl( \mu_M^G \bigl[{\D}_M^{L^{2\omega}}\bigr] \bigr)
	= (-1)^{\frac{\dim G/K}{2}} Q_{I\! V}(M_{\lambda}, \omega_{\lambda}).}
\]
If $-i\lambda$ does not lie in the image of $\Phi$, then the integer on the left hand side equals zero.
\end{theorem}

We use the compact version of quantisation to define the quantisation $Q_{I\! V}(M_{\lambda}, \omega_{\lambda})$ of the symplectic reduction, since this version is well-defined in the orbifold case.

If $G=K$, then the irreducible discrete series representation $\HH$ is the irreducible representation $V_{\lambda- \rho_c}$ of $K$ with highest weight $\lambda - \rho_c$ (see \cite{Schmid}, corollary on page 105). Hence $R^{\HH}_G$ amounts to taking the multiplicity of $V_{\lambda-\rho_c}$, as remarked after the definition of $R_G^{\HH}$. The assumption that $M/G$ is compact is now equivalent to compactness of $M$ itself.
Therefore  Theorem \ref{GSss} indeed reduces to Theorem \ref{thm [Q,R]=0 cpt spin} in this case. As mentioned before, our proof of Theorem \ref{GSss} is based on this statement for the compact case, so that we
cannot view Theorem \ref{thm [Q,R]=0 cpt spin} as a corollary to Theorem \ref{GSss}.

To obtain results about discrete series representations, we would like to apply Theorem \ref{GSss} to cases where $M$ is a coadjoint orbit of some semisimple group, such that the quantisation of this orbit in the sense of Definition \ref{def quant VI} is the $K$-theory class of a discrete series representation of this group. The condition that $M/G$ is compact rules out any interesting applications in this direction, however. If we could generalise Theorem \ref{GSss} to a similar statement where the assumption that $M/G$ is compact is replaced by the assumption that the momentum map is proper, then we might be able to deduce interesting corollaries in representation theory.

One such application could be analogous to unpublished work of Duflo and Vargas about restricting discrete series representations to semisimple subgroups. In this case, the assumption that the momentum map is proper corresponds to their assumption that the restriction map from some coadjoint orbit to the dual of the Lie algebra of such a subgroup is proper.

An interesting refinement of a special case of Duflo and Vargas's work was given by Paradan \cite{Paradan2}, who gives a multiplicity formula for the decomposition of the restriction of a discrete series representation of $G$ to $K$, in terms of symplectic reductions of the coadjoint orbit corresponding to this discrete series representation.

\section[Hamiltonian induction and cross-sections]{Induction and cross-sections of Hamiltonian group actions} \label{cpt <-> ncpt}

In this section, we explain the Hamiltonian induction and Hamiltonian cross-section constructions mentioned in the Introduction. We will see in Subsection \ref{sec inv} that they are each other's inverses.
Our term `Hamiltonian induction' is quite different from Guillemin and Sternberg's term `symplectic induction' introduced in \cite{GS2}, Section 40.

Many results in this section are known for the case where the pair $(G,K)$ is replaced by $(K,T)$. See for example \cite{Lerman,Paradan1}.

\subsection{The tangent bundle to a fibred product} \label{sec tangent}

In our study of the manifold $G \times_K N$, we will use an explicit description of its tangent bundle, which we will now explain.

For this subsection, let $G$ be any Lie group, $H<G$ any closed subgroup, and $N$ a left $H$-manifold. We consider the action of $H$ on the product $G \times N$ defined by
\[
h\cdot (g,n) = (gh^{-1}, hn),
\]
for all $h \in H$, $g \in G$ and $n \in N$. We denote the quotient of this action by $G \times_{H} N$, or by $M$. Because the action
of $H$ on $G \times N$ is proper and free, $M$ is a smooth manifold. We would like to describe the tangent bundle to $M$ explicitly.

To this end, we endow the tangent bundle $TH \cong H \times \h$ with the group structure
\[
(h, X)(h', X') := (hh', \Ad(h)X' + X),
\]
for $h, h' \in H$ and $X, X' \in \h$. This is a special case of the semidirect product
group structure on a product $V \rtimes H$, where $V$ is a representation space of $H$.
We consider the action of the group $TH$ on $TG \times TN$ defined by
\[
(h, X) \cdot (g, Y, v) := (gh^{-1}, \Ad(h)Y - X, T_n h(v) + X_{hn}),
\]
for $h \in H$, $X \in \h$, $(g, Y) \in G \times \g \cong TG$, $n \in N$ and $v \in T_nN$. Let $TG \times_{TH} TN$ be the quotient of this action. It is a vector bundle over $M$, with projection map $[g, X, v] \mapsto [g,n]$ (notation as above). We let $G$ act on $TG \times_{TH} TN$ by left multiplication on the first factor.
\begin{proposition}\label{TM}
There is a $G$-equivariant isomorphism of vector bundles
\[
\Psi: TG \times_{TH} TN \to TM,
\]
given by
\[
\Psi[g, Y, v] = Tp(g, Y, v),
\]
with $p: G \times N \to M$ the quotient map.
\end{proposition}

Now suppose that there is an $\Ad(H)$-invariant linear subspace $\p \subset \g$ such that $\g = \h \oplus \p$
(such as in the case $H = K$ we consider in the rest of this paper). Then there is a possibly simpler description of $TM$, that we will also use later.
 Consider the action of $H$ on the product $G \times TN \times \p$ given by
\[
h\cdot (g, v, Y) = (gh^{-1}, T_nh(v), \Ad(h)Y),
\]
and denote the quotient by $G \times_{H} (TN \times \p)$.
\begin{lemma} \label{TM2}
The map
\[
\Xi: TG \times_{TH} TN \to G \times_{H} (TN \times \p),
\]
given by
\[
\Xi[g,Y,v] = [g,v + (Y_{\h})_n,Y_{\p}]
\]
for all $g \in G$, $Y \in \g$, $n \in N$ and $v \in T_nN$,
is a well-defined, $G$-equivariant isomorphism of vector bundles. Here $Y_{\h}$ and $Y_{\p}$ are the components of $Y$ in $\h$ and $\p$ respectively,
according to the decomposition $\g = \h \oplus \p$.
\end{lemma}
Because of Proposition \ref{TM} and Lemma \ref{TM2}, we have $TM \cong G \times_{H} (TN \times \p)$ as $G$-vector
bundles.\footnote{A version of this fact is used without a proof in \cite{AS} on page 503.}

In Section \ref{sec preqSpin}, we will use the following version of Proposition \ref{TM} and Lemma \ref{TM2}.
\begin{corollary} \label{cor TM}
In the situation of Lemma \ref{TM2}, there is an isomorphism of $G$-vector bundles
\[
TM \cong \bigl(p_{G/H}^*T(G/H)\bigr) \oplus (G \times_H TN),
\]
where $p_{G/H}: M \to G/H$ is the natural projection.
\end{corollary}
\begin{proof}
The claim follows from  Proposition \ref{TM}, Lemma \ref{TM2}, and the fact that
\[
T(G/H) \cong G \times_H \p,
\]
where $H$ acts on $\p$ via $\Ad$.
\end{proof}

\subsection{Hamiltonian induction} \label{sec ind}

We return to the standard situation in this paper, where $G$ is a semisimple group, and $K<G$ is a maximal compact subgroup.

\subsubsection*{The symplectic manifold}

Let $(N, \nu)$ be a symplectic manifold on which $K$ acts in Hamiltonian fashion, with momentum map $\Phi^N: N \to \kk^*$.
Suppose that the image of $\Phi^N$ lies in the set $\kse$, defined in \eqref{eq def kse}.
As in Subsection \ref{sec tangent}, we consider the fibred product $M = G \times_K N$, equipped with the action of $G$ induced by left multiplication on the first factor. As a consequence of Proposition \ref{TM} and Lemma \ref{TM2}, we have for all $n \in N$,
\[
T_{[e,n]}M \cong   T_nN \oplus \p.
\]
We define a two-form $\omega$ on $M$ by requiring that it is $G$-invariant, and that for all $X,Y \in \p$, $n \in N$ and $v, w \in T_nN$,
\begin{equation} \label{eq omega}
\omega_{[e,n]}(v+X,w+ Y) := \nu_n(v,w) - \langle \Phi^N(n), [X,Y]\rangle.
\end{equation}
Note that $[X,Y] \in \kk$ for all $X,Y \in \p$, so the pairing in the second term is well-defined. We claim that $\omega$ is a symplectic form. This is analogous to formula (7.4) from \cite{Paradan1}.
\begin{proposition} \label{prop omega sympl}
The form $\omega$ is symplectic.
\end{proposition}

\subsubsection*{The momentum map}

Next, consider the map
$\Phi^M: M \to \g^*$
given by
\begin{equation} \label{eq Phi}
\Phi^M[g,n] = \Ad^*(g) \Phi^N(n).
\end{equation}
This map is well-defined by $K$-equivariance of $\Phi^N$. Furthermore, it is obviously $G$-equivariant, and its image lies in $\gse$.
\begin{proposition} \label{prop Phi mom}
The map $\Phi^M$ is a momentum map for the action of $G$ on $M$.
\end{proposition}

\begin{definition} \index{Hamiltonian induction}
The \emph{Hamiltonian induction}\index{Hamiltonian induction} of the Hamiltonian action of $K$ on $(N, \nu)$ is the Hamiltonian action of $G$ on $(M, \omega)$:\index{haa@$\HInd_K^G$}
\[
\HInd_K^G(N, \nu, \Phi^N) := (M, \omega, \Phi^M).
\]
\end{definition}
\begin{example}
Let $\xi \in \torus^* \setminus \ncw$ be given, and consider the coadjoint orbit $N := K \cdot \xi \subset \kk^*$. The Hamiltonian induction of the coadjoint action of $K$ on $N$ is the coadjoint action of $G$ on the coadjoint orbit $M := G \cdot \xi$, including the natural symplectic forms and momentum maps. Indeed, the map
\[
 G \cdot \xi \to G \times_K N
\]
given by $g\cdot \xi \mapsto [g, \xi]$ is a symplectomorphism.
\end{example}

\subsection{Hamiltonian cross-sections} \label{sec res}

We now turn to the inverse construction to Hamiltonian induction, namely the \emph{Hamiltonian cross-section}.
In this case, we start with a Hamiltonian $G$-manifold $(M, \omega)$, with momentum map $\Phi^M$.
Such a cross-section will indeed be symplectic and carry a Hamiltonian $K$-action, under the assumption that the image of $\Phi^M$ is contained in $\gse$.
A Hamiltonian cross-section is a kind of double restriction: it is both a restriction to a subgroup of $G$ and a restriction to a submanifold of $M$.

Most of this subsection is based on the proof of the symplectic cross-section theorem in Lerman et al.\ \cite{Lerman}. We will therefore omit most proofs.

As before, we identify $\kk^*$ with the subspace $\p^0$ of $\g^*$.
The main result of this subsection is:
\begin{proposition} \label{prop Hamres}
If $\Phi^M(M) \subset \gse$, then $N:= \bigl(\Phi^M\bigr)^{-1}(\kk^*)$ is a $K$-invariant symplectic submanifold  of $M$, and $\Phi^N:= \Phi^M|_N$ is a momentum map for the action of $K$ on $N$.
\end{proposition}
We denote the restricted symplectic form $\omega|_N$ by $\nu$.
\begin{definition} \index{Hamiltonian cross-section}
The \emph{Hamiltonian cross-section}\index{Hamiltonian cross-section} of the Hamiltonian action of $G$ on $(M, \omega)$ is the Hamiltonian action of $K$ on $(N, \nu)$: \index{haa@$\HCross_K^G$}
\[
\HCross^G_K(M, \omega, \Phi^M) := (N, \nu, \Phi^N).
\]
\end{definition}
In Proposition \ref{prop Ind Res inv}, we will see that $M \cong G \times_K N$, so that $M/G$ is compact if and only if $N$ is.

To prove Proposition \ref{prop Hamres}, we have to show that $N$ is a smooth submanifold of $M$, and that the restricted form $\omega|_N$ is symplectic. Then the submanifold $N$ is $K$-invariant by $K$-equivariance of $\Phi^M$, and the fact that $\Phi^N$ is a momentum map is easily verified. We begin with some preparatory lemmas, based on the proof of the symplectic cross-section theorem mentioned above.

For the remainder of this subsection, let $m \in M$ be given, and write $\xi := \Phi^M(m)$.
\begin{lemma} \label{lem psi}
The linear map
\[
\psi: T_m (G \cdot m) \to T_{\xi}(G \cdot \xi)
\]
given by
\[
\psi(X_m) = X_{\xi}
\]
for $X \in \g$, is symplectic, in the sense that for all $X, Y \in \g$,
\[
\omega_m(X_m, Y_m) = -\langle \xi, [X, Y]\rangle.
\]
\end{lemma}

\begin{lemma} \label{lem incl}
We have the following inclusions of subspaces of $\g^*$:
\[
\g_{\xi}^0 \subset T_m\Phi^M(T_mM) \subset \g_m^0.
\]
\end{lemma}

\begin{lemma} \label{lem pm}
If $m \in N \subset M$, then the subspace
\[
\p \cdot m := \{X_m; X \in \p\} \subset T_mM
\]
is symplectic.
\end{lemma}
\begin{proof}
\emph{Step 1:} we have
\[
T_{\xi}(G \cdot \xi) \cong \g \cdot \xi = (\kk + \p)\cdot \xi = T_{\xi}(K \cdot \xi) + \p \cdot \xi.
\]

\emph{Step 2:} the subspace $\p \cdot \xi \subset T_{\xi}(G\cdot \xi)$ is symplectic.

\noindent
Indeed, by Step 1 and Lemma \ref{lem sympl subsp} below, it is enough to prove that $\p \cdot \xi$ and $T_{\xi}(K\cdot \xi)$ are symplectically orthogonal. Let $X \in \kk$ and $Y \in \p$ be given. Because $m \in N$, we have $\xi \in \kk^*$, and also $\ad^*(X)\xi \in \kk^* \cong \p^0$. Hence
\[
\langle \xi, [X,Y] \rangle = -\langle \ad^*(X) \xi, Y\rangle = 0.
\]

\emph{Step 3:} the subspace $\p \cdot m \subset T_mM$ is symplectic.

\noindent
Indeed, let a nonzero  $X \in \p$ be given. We are looking for a $Y \in \p$ such that $\omega_m(X_m, Y_m)\not=0$.
Note that by Lemma \ref{lem gxi p}, we have $\ad^*(X)\xi = X_{\xi} \not=0$. So by Step 2, there is a $Y \in \p$ for which $\langle \xi, [X,Y]\rangle \not= 0$. Hence by Lemma \ref{lem psi},
\[
\omega_m(X_m, Y_m) = -\langle \xi, [X,Y]\rangle \not=0.
\]
\end{proof}
In Step 2 of the proof of Lemma \ref{lem pm}, we used
\begin{lemma} \label{lem sympl subsp}
Let $(W, \sigma)$ be a symplectic vector space, and let $U,V \subset W$ be linear subspaces. Suppose that $W = U + V$, and that
$U$ and $V$ are symplectically orthogonal. Then $U$ and $V$ are \emph{symplectic} subspaces.
\end{lemma}

After these preparations, we are ready to prove Proposition \ref{prop Hamres}.

\medskip
\noindent \emph{Proof of Proposition \ref{prop Hamres}.}
We first show that $N$ is smooth. This is true if $\Phi^M$ satisfies the transversality condition that for all $n \in N$, with $\eta := \Phi^M(n)$, we have
\[
T_{\eta} \g^* = T_{\eta}\kk^* + T_n\Phi^M(T_nM).
\]
(See e.g.\ \cite{Hirsch}, Chapter 1, Theorem 3.3.)
 By Lemma \ref{lem incl}, we have $\g_{\eta}^0 \subset T_n\Phi^M(T_nM)$, and by Lemma \ref{lem gxi p}, we have $\g_{\eta} \cap \p = \{0\}$.
Now, using the fact that $V^0 + W^0 = (V \cap W)^0$ for two linear subspaces $V$ and $W$ of a vector space, we see that
\[
 T_{\eta}\kk^* + T_n\Phi^M(T_nM) \supset \p^0 + \g_{\eta}^0 = (\p \cap \g_{\eta})^0 = \{0\}^0 = \g^*.
\]
This shows that $N$  is indeed smooth.

Next, we prove that $\omega|_N$ is a symplectic form. It is closed because $\omega$ is, so it remains to show that it is nondegenerate. Let $n \in N$ be given. By Lemma \ref{lem sympl subsp}, it is enough to show that $T_nM = T_nN + \p \cdot n$, and that $T_nN$ and $\p \cdot n$ are symplectically orthogonal.

We prove that $T_nM = T_nN \oplus \p \cdot n$, by first noting that
\[
\dim N = \dim M - \dim \g^* + \dim \kk^* = \dim M - \dim \p.
\]
Because $\g_n \subset \g_{\Phi^M(n)}$, and $\g_{\Phi^M(n)} \cap \p = \{0\}$ by Lemma \ref{lem gxi p},
we have $\dim \p = \dim (\p \cdot n)$, and
\[
\dim T_nM = \dim T_nN + \dim (\p \cdot n).
\]
It is therefore enough to prove that $T_nN \cap \p \cdot n = \{0\}$. To this end, let $X \in \p$ be given, and suppose $X_n \in T_nN$. That is, $T_n\Phi^M(X_n) \in \kk^*$, which is to say that for all $Y \in \p$,
\[
\omega_{n}(X_n, Y_n) = -\langle T_n\Phi^M(X_n), Y \rangle = 0.
\]
By Lemma \ref{lem pm}, it follows that $X_n = 0$, so that indeed $T_nN \cap \p \cdot n = \{0\}$.

Finally, we show that for all $v \in T_nN$ and $X \in \p$, we have $\omega_n(v, X_n)=0$.
Indeed, for such $v$ and $X$, we have $T_n\Phi^M(v) \in \kk^* \cong \p^0$, so
\[
\omega_n(v,X_n) = \langle T_n\Phi^M(v), X\rangle = 0.
\]
\hfill $\square$

\subsection[Induction and cross-sections are inverse]{Hamiltonian induction and taking Hamiltonian cross-sections are mutually inverse} \label{sec inv}

Let us prove the statement in the title of this subsection. One side of it (Proposition \ref{prop Ind Res inv}) will be used in the proof of Theorem \ref{GSss} in Subsection \ref{sec [Q,I] = [Q,R]}. We will not use the other side (Proposition \ref{prop Res Ind inv}).

\subsubsection*{Induction of a cross-section}

First, we have
\begin{proposition} \label{prop Ind Res inv}
Let $(M, \omega, \Phi^M)$ and $(N, \nu, \Phi^N) := \HCross_K^G(M, \omega, \Phi^M)$ be as in Subsection \ref{sec res}. Consider the manifold $\widetilde{M}:= G \times_K N$, with symplectic form $\tilde \omega$ equal to the form $\omega$ in \eqref{eq omega}. Define the map $\widetilde{\Phi^M}$ as the map $\Phi^M$ in \eqref{eq Phi}. Then the map
\[
\varphi: \widetilde{M} \to M
\]
given by
\[
\varphi[g,n] = g \cdot n
\]
is a well-defined, $G$-equivariant symplectomorphism, and $\varphi^*\Phi^M = \widetilde{\Phi^M}$.

Put differently, $\HInd_K^G \circ \HCross_K^G$ is the identity, modulo equivariant symplectomorphisms that intertwine the momentum maps.
\end{proposition}
It follows from this proposition that $M/G = N/K$, so that $M/G$ is compact if and only if $N$ is compact.
\begin{proof}
The statement about the momentum maps follows from $G$-equivariance of $\Phi^M$.

The map $\varphi$ is well-defined by definition of the action of $K$ on $G \times N$.
It is obviously $G$-equivariant. Furthermore, $\varphi$ is smooth because the action of $G$ on $M$ is smooth (this was a tacit assumption),
and by definition of the smooth structure on the quotient $G \times_K N$.

To prove injectivity of $\varphi$, let $g, g' \in G$ and $n, n' \in N$ be given, and suppose that $g \cdot n = g' \cdot n'$.
Because $\Phi^M(N) \subset \kse$, there are $k, k' \in K$ and $\xi, \xi' \in \torus^*_+ \setminus \ncw$ such that
\[
\begin{split}
\Phi^M(n) &= k\cdot \xi ; \\
\Phi^M(n') &= k' \cdot \xi'.
\end{split}
\]
Then by equivariance of $\Phi^M$, we have $gk \cdot \xi = g'k' \cdot \xi'$.
Because $\torus^*_+ \setminus \ncw$ is a fundamental domain for the coadjoint action of $G$ on $\gse$, we must have $\xi = \xi'$, and
\[
k'^{-1}g'^{-1}gk \in G_{\xi} \subset K.
\]
So $k'' := g'^{-1}g \in K$. Hence
\[
g'k''n = g \cdot n = g'\cdot n',
\]
and $k''\cdot n = n'$. We conclude that
\[
[g', n'] = [gk''^{-1}, k'' \cdot n] = [g,n],
\]
and $\varphi$ is injective.

To prove surjectivity of $\varphi$, let $m \in M$ be given. Since $\Phi^M(m) \in \gse$, there are $g \in G$ and $\xi \in \torus^*_+ \setminus \ncw$ such that $\Phi^M(m) = g\cdot \xi$. Set $n := g^{-1}m$. Then $\Phi^M(n) = \xi \in \kk^*$, so $n \in N$, and $\varphi[g,n]=m$.

Next, we show that the inverse of $\varphi$ is smooth. We prove this using the inverse function theorem: smoothness of $\varphi^{-1}$
follows from the fact that the tangent map $T\varphi$ is invertible. Or, equivalently, from the fact that the map $\widetilde{T\varphi}$,
 defined by the following diagram, is invertible.
\[
\xymatrix{T(G \times_K N) \ar[r]^-{T\varphi} & TM \\
	TG \times_{TK} TN. \ar[u]^{\Psi}_{\cong} \ar[ur]_-{\widetilde{T\varphi}}}
\]
Here $\Psi$ is the isomorphism from Proposition \ref{TM}. Explicitly, the map $\widetilde{T\varphi}$ is given  by
\[
\begin{split}
\widetilde{T\varphi}[g,X,v] &= T\varphi \circ Tp(g,X,v) \\
	&= T\alpha(g,X,v),
\end{split}	
\]
for all $g \in G$, $X \in \g$ and $v \in T_nN$,
with $\alpha: G \times N \to M$ the action map. Let $\gamma$ be a curve in $N$ with $\gamma(0)=n$ and $\gamma'(0)=v$. Then we find that
\begin{equation} \label{eq Talpha}
\begin{split}
T\alpha(g,X,v) &= \ddt \exp(tX)g \cdot \gamma(t) \\
	&= X_{gn} + T_{n}g(v).
\end{split}
\end{equation}

Because the vector bundles $TG \times_{TK} TN$ and $TM$ have the same rank, it is enough to show that $\widetilde{T\varphi}$
 is surjective. To this end, let $m \in M$ and $w \in T_mM$ be given. Since $\varphi$ is surjective, there are $g \in G$ and $n \in N$ such that $m=g\cdot n$. Furthermore, we have
\[
T_nM = T_nN + \g \cdot n.
\]
Indeed, in our situation we even have $T_nM = T_nN \oplus \p \cdot n$ (see the proof of Proposition \ref{prop Hamres}). Hence
\[
T_mM = T_n g(T_nM) = T_ng(T_nN + \g \cdot n).
\]
Therefore, there are $v \in T_nN$ and $X \in \g$ such that
\[ \begin{split}
w &= T_ng(v + X_n) \\
	&= T_ng(v) + \bigl(\Ad(g)X \bigr)_{g \cdot n} \\
	&= \widetilde{T\varphi}[g,\Ad(g)X,v],
\end{split} \]
by \eqref{eq Talpha}. This shows that $\widetilde{T\varphi}$ is indeed surjective.

Finally, we prove that $\varphi$ is a symplectomorphism. Let $n \in N$, $v,w \in T_nN$ and $X,Y \in \p$  be given. We will show that
\[
\omega_n \bigl(T_{[e,n]}\varphi(v+X), T_{[e,n]}\varphi(w+Y) \bigr) = \omega_n(v,w) - \langle \Phi^M(n), [X,Y]\rangle.
\]
By $G$-invariance of the symplectic forms $\omega$ and $\tilde \omega$, this implies that $\varphi$ is a symplectomorphism on all of $\widetilde{M}$.

Similarly to \eqref{eq Talpha}, we find that $T_{[e,n]}\varphi(v+X) = v + X_n$.
Therefore,
\begin{align}
\omega_n \bigl(T_{[e,n]}\varphi(v+X), T_{[e,n]}\varphi(w+Y) \bigr) &= \omega_n(v+X_n, w+Y_n) \nonumber \\
	&= \omega_n(v, w) + \omega_n(X_n, Y_n) , \label{eq sympl forms}
\end{align}
since $T_nN$ and $\p \cdot n$ are symplectically orthogonal (see the end of the proof of Proposition \ref{prop Hamres}). Now applying Lemma \ref{lem psi} to the first term in \eqref{eq sympl forms} gives the desired result.
\end{proof}

\subsubsection*{Cross-section of an induction}

Conversely to Proposition \ref{prop Ind Res inv}, we have:
\begin{proposition} \label{prop Res Ind inv}
Let $(N, \nu, \Phi^N)$ and $(M, \omega, \Phi^M) := \HInd(N, \nu, \Phi^N)$ be as in Subsection \ref{sec ind}. Suppose $\Phi^N(N) \subset \kse$. Then
\[
(N, \nu) \cong \Bigl(\bigl(\Phi^M\bigr)^{-1}(\kk^*), \omega|_{\left(\Phi^M\right)^{-1}(\kk^*)}\Bigr),
\]
and this isomorphism intertwines the momentum maps $\Phi^N$ and $\Phi^M$.

In other words, $\HCross^G_K \circ \HInd^G_K$ is the identity, modulo equivariant symplectomorphisms that intertwine the momentum maps.
\end{proposition}
\begin{proof}
We claim that
\begin{equation} \label{tilde N}
\bigl(\Phi^M\bigr)^{-1}(\kk^*) = \{[e,n]; n  \in N\} =: \widetilde{N}.
\end{equation}
The map $n \mapsto [e,n]$ is a diffeomorphism from $N$ to $\widetilde{N}$. It is clear that this diffeomorphism is $K$-equivariant, and intertwines the momentum maps $\Phi^N$ and $\Phi^M$.

To prove that $\bigl(\Phi^M\bigr)^{-1}(\kk^*) = \widetilde{N}$, let $[g,n] \in M$ be given, and suppose $\Phi^M[g,n] = g\cdot \Phi^N(n) \in \kk^*$. Because $\Phi^N(N) \subset \kse$, we have
\[
g \cdot \Phi^N(n) \in \bigl(G \cdot \kse\bigr) \cap \kk^* = \kse.
\]
So there are $k, k' \in K$ and $\xi, \xi' \in \torus^*_+ \setminus \ncw$ such that
\[
\begin{split}
\Phi^N(n) &= k\cdot \xi;\\
g\cdot \Phi^N(n) &= k'\cdot \xi'.
\end{split}
\]
Hence $gk\cdot \xi = k'\cdot \xi'$, and since $\torus^*_+ \setminus \ncw$ is a fundamental domain for the coadjoint action of $G$ on $\gse$, we have $\xi'=\xi$. So
\[
k'^{-1}gk \in G_{\xi} \subset K,
\]
and hence $g \in K$. We conclude that $[g,n] = [e, g^{-1}n]$, which proves \eqref{tilde N} (the inclusion $\widetilde{N} \subset \bigl(\Phi^M\bigr)^{-1}(\kk^*)$ follows from the definition of $\Phi^M$).

For each $n \in N$, the natural isomorphism $v \mapsto [e,0,v]$ from $T_nN$ to $T_{[e,n]} \widetilde{N}$ intertwines the respective symplectic forms, by definition of those forms.
\end{proof}

\section[Induction of other structures]{Induction of prequantisations and $\Spin^c$-structures}\label{sec preqSpin}

We extend the induction procedure of Section \ref{cpt <-> ncpt} to prequantisations and to  $\Spin^c$-structures, used to define quantisation. For prequantisations, it is possible to define restriction to a Hamiltonian cross-section
in a suitable way.  For our purposes, it is not necessary to restrict $\Spin^c$-structures.

\subsection{Prequantisations} \label{sec preq} \index{prequantisation!induced --}

Since we are interested in quantising Hamiltonian actions, let us look at induction  of prequantum line bundles, and at restriction to Hamiltonian cross-sections.

\subsubsection*{Restriction to Hamiltonian cross-sections}

The easy part is restriction. Indeed, let $(M, \omega)$ be a Hamiltonian $G$-manifold, let $\Phi^M$ be a momentum map with $\Phi^M(M) \subset \gse$,
and let $(N, \nu, \Phi^N)$ be the Hamiltonian cross-section of this action. Now let ${L^{\omega}} \to M$ be a prequantum line bundle, let $(\relbar, \relbar)_{{L^{\omega}}}$ be a $G$-invariant
Hermitian metric on ${L^{\omega}}$,  and let $\nabla^M$ be a $G$-equivariant Hermitian connection
on ${L^{\omega}}$ with curvature $2\pi i \, \omega$. Let $\nabla^N$ be the connection on ${L^{\nu}} := {L^{\omega}}|_N$ defined as the pullback of $\nabla^M$ along the inclusion map $N \hookrightarrow M$. It is given by
\[
\nabla^N\bigl(s|_N\bigr) = \left(\nabla^Ms \right)|_N,
\]
for all sections $s \in \Gamma^{\infty}({L^{\omega}})$. This is indeed a connection, with curvature
\[
R_{\nabla^N} = R_{\nabla^M}|_N = 2\pi i \,  \omega|_N = 2 \pi i \, \nu.
\]
Furthermore, it is Hermitian with respect to the restriction \mbox{$(\relbar, \relbar)_{{L^{\nu}}}$} of \mbox{$(\relbar, \relbar)_{{L^{\omega}}}$.}
That is, $({L^{\nu}}, (\relbar, \relbar)_{{L^{\nu}}}, \nabla^N)$ is a prequantisation of the action of $K$ on $N$.

In the same way, we see that a $\Spin^c$-prequantum line bundle on $(M, \omega)$, that is, a prequantum line bundle on $(M, 2\omega)$, restricts to a $\Spin^c$-prequantum line bundle on $(N, 2\nu)$.

\subsubsection*{Induction: an auxiliary connection $\nabla$}

Now let us consider induction of prequantisations. As in Subsection \ref{sec ind}, let $(N, \nu)$ be a Hamiltonian $K$-manifold,
with momentum map $\Phi^N$.  Let $(M, \omega, \Phi^M)$ be the Hamiltonian induction of these data.
Let $\bigl({L^{\nu}}, (\relbar, \relbar)_{{L^{\nu}}}, \nabla^{N}\bigr) $ be an equivariant prequantisation of the action of $K$ on $N$.
As in the case of restriction, the following argument extends directly to $\Spin^c$-prequantisations.

Consider the line bundle
\[
{L^{\omega}} := G \times_K {L^{\nu}} \to M,
\]
with the natural projection map $[g,l] \mapsto [g,n]$ for $g \in G$, $n \in N$ and $l \in L^{\nu}_n$.
Let $(\relbar, \relbar)_{{L^{\omega}}}$ be the $G$-invariant Hermitian metric on ${L^{\omega}}$ induced by \mbox{$(\relbar, \relbar)_{{L^{\nu}}}$:} for all $g, g' \in G$, $n \in N$ and $l, l' \in L^{\nu}_n$, set
\[
\bigl([g,l], [g',l']\bigr)_{{L^{\omega}}} := (l,l')_{{L^{\nu}}}.
\]
In the remainder of this subsection, we will construct a connection $\nabla^M$ on ${L^{\omega}}$, such that $\bigl({L^{\omega}}, (\relbar, \relbar)_{{L^{\omega}}}, \nabla^M \bigr)$
 is a $G$-equivariant prequantisation of $(M, \omega)$. This is by definition the prequantisation induced by $\bigl({L^{\nu}}, (\relbar, \relbar)_{{L^{\nu}}}, \nabla^{N}\bigr)$.

To construct the connection $\nabla^M$, we consider the line bundle
\[
L := G \times {L^{\nu}} \to G \times N,
\]
with the obvious projection map $(g,l) \mapsto (g,n)$, for all $g \in G$, $l \in L^{\nu}_n$. Then ${L^{\omega}} = L/K$, where $K$ acts on $L$ by
\[
k \cdot(g,l) = (gk^{-1},  k \cdot l),
\]
for $k \in K$, $g \in G$ and $l \in {L^{\nu}}$. We therefore have a linear isomorphism
\[
\psi_L: \Gamma^{\infty}(L)^K \to \Gamma^{\infty}({L^{\omega}}),
\]
given by
\begin{equation} \label{eq psiLM}
\psi_L(\sigma)[g,n] = [\sigma(g,n)].
\end{equation}
We will construct $\nabla^M$ as the connection induced by a $K$-equivariant connection $\nabla$ on $L$.
The space $\Gamma^{\infty}(L)$ of sections of $L$ is isomorphic to the space
\[
\widetilde{\Gamma}^{\infty}(L) := \{s: G \times N \xrightarrow{C^{\infty}} {L^{\nu}}; s(g,n) \in L^{\nu}_n \text{ for all $g \in G$ and $n \in N$}.\}
\]
Indeed, the isomorphism is given by $s \mapsto \sigma$, where $\sigma(g,n)  =(g, s(g, n))$. For $s \in \widetilde{\Gamma}^{\infty}(L)$, $g \in G$ and $n \in N$, we write
\[
s_g(n) := s(g,n) =: s^n(g).
\]
(We will use the same notation when $s$ is replaced by a function on $G \times N$.) Then for fixed $g$, $s_g$ is a section of ${L^{\nu}}$, and for fixed $n$, $s^n$ is a function
\[
s^n: G \to L^{\nu}_n.
\]

Let $s \in \widetilde{\Gamma}^{\infty}(L)$, $X \in \g$,  $v \in \mathfrak{X}(N)$, $g \in G$ and $n \in N$ be given. We define
\begin{equation} \label{eq nabla}
\bigl(\nabla_{v+X}s\bigr)(g,n) := \bigl(\nabla^N_v s_g\bigr)(n) + X(s^n)(g) + 2\pi i \, \Phi^N_{X_{\kk}}(n) s(g,n).
\end{equation}
Here we have written $X = X_{\kk} + X_{\p} \in \kk \oplus \p$. (The subscript $\kk$ in $X_{\kk}$ in \eqref{eq nabla} is actually superfluous, because we identify $\kk^*$ with $\p^0 \subset \g^*$.) In the expression $X(s^n)$, we view $X$ as a left invariant vector field on $G$, acting on the function $s^n$. Note that all tangent vectors in $T_{(g, n)}(G \times N)$ are of the form $X_g+v_n = (g, X, v_n) \in T_gG \times T_nN$, and therefore the above formula determines $\nabla$ uniquely. We claim that $\nabla$ is a $K$-equivariant connection on $L$ with the right curvature, so that it induces a connection $\nabla^M$ on ${L^{\omega}}$ with curvature $\omega$.

\begin{lemma} \label{lem nabla conn}
The formula \eqref{eq nabla} defines a connection $\nabla$ on $L$.
\end{lemma}
\begin{proof}
The Leibniz rule for $\nabla$ follows from the fact that for $f \in C^{\infty}(G \times N)$, $X \in \g$,  $v \in \mathfrak{X}(N)$, $g \in G$ and $n \in N$, one has
\[
(v+X)(f)(g,n) = v(f_g)(n) + X(f^n)(g).
\]

Linearity over $C^{\infty}(G \times N)$ in the vector fields follows from the fact that, with notation as above,
\[
\bigl(f(v+X) \bigr)_{(g,n)} = \bigl(f^nX \bigr)_g + \bigl(f_gv \bigr)(n).
\]

Locality is obvious.
\end{proof}

\subsubsection*{Properties of the connection $\nabla$}

Let $(\relbar, \relbar)_L$ be the Hermitian metric on $L$ given by
\[
\bigl((g,l), (g',l')\bigr)_L := (l,l')_{{L^{\nu}}}
\]
for all $g,g' \in G$ and $l,l' \in L^{\nu}_n$.
\begin{lemma} \label{lem nabla Herm}
The connection $\nabla$ is Hermitian with respect to this metric.
\end{lemma}

Next, we compute the curvature of $\nabla$. This computation is long but straightforward.
\begin{lemma} \label{lem curv nabla}
The curvature $R_{\nabla}$ of $\nabla$ is given by
\[
R_{\nabla}(v+X, w+Y)(g,n) = 2\pi i \bigl(\nu_n(v,w) - \langle \Phi^N(n), [X,Y]_{\kk} \rangle \bigr),
\]
for all $X, Y \in \g$,  $v, w \in \mathfrak{X}(N)$, $g \in G$ and $n \in N$.
\end{lemma}

It remains to show that the connection $\nabla$ induces the desired connection $\nabla^M$ on ${L^{\omega}}$. This will follow from $K$-equivariance of $\nabla$.
\begin{lemma}
The connection $\nabla$ is $K$-equivariant in the sense that for all
$X \in \g$, $v \in \mathfrak{X}(N)$, $k \in K$, $s \in \Gamma^{\infty}(L)$, $g \in G$ and $n \in n$, we have
\[
k\cdot \bigl(\nabla_{v+X} s\bigr)  = \nabla_{k\cdot(v+X)} k\cdot s.
\]
\end{lemma}

We now define $\nabla^M$ via the isomorphism $\psi_L$ in \eqref{eq psiLM}. Note that by Proposition \ref{TM} and Lemma \ref{TM2}, we have
\[
\begin{split}
\X(M) &\cong \Gamma^{\infty}(G \times_K N, G \times_K(TN \times \p)) \\
	&\cong \Gamma^{\infty}(G \times N, G \times TN \times \p)^K \\
	&\subset \Gamma^{\infty}(G \times N, (G \times \g) \times TN )^K \\
	&= \X(G \times N)^K
\end{split}
\]
We will write $j: \X(M) \hookrightarrow \X(G \times N)^K$ for this embedding map. For $w \in \X(M)$ and $s \in \Gamma^{\infty}(L)^K$, we define the connection $\nabla^M$ by
\[
\nabla^M_w \psi_L(s) := \psi_L \bigl(\nabla_{j(w)}s\bigr).
\]
Because $s$ and $j(w)$ are $K$-invariant, and $\nabla$ is $K$-equivariant, we indeed have $\nabla_{j(w)}s \in \Gamma^{\infty}(L)^K$, the domain of $\psi_L$.

It now follows directly from the definitions and from Lemmas \ref{lem nabla conn}, \ref{lem nabla Herm} and \ref{lem curv nabla}
that $\nabla^M$ is a Hermitian connection on ${L^{\omega}}$ with curvature $\omega$.

\subsubsection*{Induction and restriction}

The induction and restriction procedures for line bundles described above are each other's inverses (modulo equivariant line bundle isomorphisms),
although this does not include the connections on the bundles in question:
\begin{lemma} \label{lem lbindres}
(i) Let $N$ be a $K$-manifold, and $q^N: E^N \to N$ a $K$-vector bundle. Then
\[
\bigl(G\times_K E^N\bigr)|_{\widetilde{N}} \cong E^N,
\]
with $\widetilde{N}$ as in \eqref{tilde N}.

(ii) Let $M$ be a $G$-manifold, $E^M \to M$ a $G$-vector bundle. Let $N \subset M$ be a $K$-invariant submanifold, and denote the restriction of $E^M$ to $N$ by $E^N$. Let $\varphi: G \times_K N \to M$ be the map $\varphi[g,n] = gn$. Then
\[
\varphi^*E^M \cong G \times_K E^N.
\]
\end{lemma}
\begin{proof}
\emph{(i)} Note that
\[
\begin{split}
\bigl(G\times_K E^N\bigr)|_{\widetilde{N}} &= \bigl\{[g,v] \in G \times_K E^N; [g, q^N(v)] = [e,n] \text{ for an $n \in N$} \bigr\} \\
	&= \bigl\{[e,v] \in G \times_K E^N; v \in E^N \bigr\} \\
	&\cong E^N.
\end{split}
\]

\emph{(ii)} Note that
\[
\varphi^*E^M = \bigl\{ \bigl([g,n], v \bigr); g \in G, n \in N \text{ and $v \in E^M_{gn}$} \bigr\}. \\
\]
The map $\bigl([g,n], v \bigr) \mapsto [g,v]$ is the desired vector bundle isomorphism onto $G \times_K E^N$.
\end{proof}
For our purposes, it does not matter that this lemma says nothing about connections that may be defined on the vector bundles in question,
because the $K$-homology classes defined by Dirac operators associated to such connections are homotopy invariant. In our setting, the vector bundle isomorphisms in the proof of Lemma \ref{lem lbindres} do intertwine the metrics $(\relbar, \relbar)_{{L^{\omega}}}$ and $(\relbar, \relbar)_{{L^{\nu}}}$ on the respective line bundles.

\subsection{$\Spin^c$-structures} \label{sec Spinc}

Because we want to compare the Dirac operators on $M$ and $N$, we now look at induction of $\Spin^c$-structures. As before, we consider a semisimple group $G$ with maximal compact subgroup $K$, and a $K$-manifold $N$. We form the fibred product $M := G \times_K N$,
and we will show how a $K$-equivariant $\Spin^c$-structure on $N$ induces a $G$-equivariant $\Spin^c$-structure on $M$. It will turn out that the operation of taking determinant line bundles intertwines the induction process for $\Spin^c$-structures in this subsection, and the induction process for prequantum line bundles in the previous one.

\subsubsection*{General constructions}

The construction of induced $\Spin^c$-structures we will use, is based on the following two facts, of which we were informed by Paul-\'Emile Paradan.
\begin{lemma} \label{lem Spinc sum}
For $j=1,2$, let $E_j \to M$ be a real vector bundle over a manifold $M$. Suppose $E_1$ and $E_2$ are equipped with metrics and orientations. Let $P_j \to M$ be a $\Spin^c$-structure on $E_j$, with determinant line bundle $L_j \to M$. Then there is a $\Spin^c$-structure $P \to M$ on the direct sum $E_1 \oplus E_2 \to M$, with determinant line bundle $L_1 \otimes L_2$.
\end{lemma}
\begin{proof}
Let $r_j$ be the rank of $E_j$, and write $r:= r_1 + r_2$. Consider the double covering map
\[
\pi: \Spin^c(r) \to \SO(r) \times \U(1),
\]
given by $[a,z] \mapsto (\lambda(a),z)$, where $a \in \Spin(r)$, $z \in \U(1)$, and $\lambda: \Spin(r) \to \SO(r)$ is the standard double covering. Consider the subgroups
\[
H':= \SO(r_1) \times \SO(r_2) \times \U(1)
\]
of $\SO(r) \times \U(1)$, and $H := \pi^{-1}(H')$ of $\Spin^c(r)$. Noting that
\[
H'\cong (\SO(r_1) \times \U(1)) \times_{\U(1)} (\SO(r_2) \times \U(1)),
\]
we see that
\[
H \cong \Spin^c(r_1) \times_{\U(1)} \Spin^c(r_2).
\]

Let $P_1 \times_{\U(1)} P_2$ be the quotient of $P_1 \times P_2$ by the $\U(1)$-action given by
\[
z(p_1, p_2) = (p_1 z, p_2 z^{-1}),
\]
for $z \in \U(1)$ and $p_j \in P_j$. Define
\[
P := \bigl(P_1 \times_{\U(1)} P_2 \bigr) \times_H \Spin^c(r).
\]
Then we have naturally defined isomorphisms
\[
\begin{split}
P \times_{\Spin^c(r)} \R^r &\cong  \bigl(P_1 \times_{\U(1)} P_2 \bigr) \times_H (\R^{r_1} \oplus \R^{r_2}) \\
	&\cong \bigl(P_1 \times_{\Spin^c(r_1)} \R^{r_1} \bigr) \oplus \bigl(P_2 \times_{\Spin^c(r_2)} \R^{r_2} \bigr) \\
	&\cong E_1 \oplus E_2.
\end{split}
\]

The determinant line bundle of $P$ is
\[
\det(P) =  \bigl(P_1 \times_{\U(1)} P_2 \bigr) \times_H \C,
\]
where $H$ acts on $\C$ via the determinant homomorphism. Note that, for all $h = [h_1, h_2] \in \Spin^c(r_1) \times_{\U(1)} \Spin^c(r_2) \cong H$, we have $\det(h) = \det(h_1)\det(h_2)$. Using this equality, one can check that the map
\[
\bigl(P_1 \times_{\U(1)} P_2 \bigr) \times_H \C \to \bigl(P_1 \times_{\Spin^c(r_1)} \C \bigr) \otimes   \bigl(P_2 \times_{\Spin^c(r_2)} \C \bigr),
\]
given by
\[
[p_1, p_2, z] \mapsto [p_1, z] \otimes [p_2, 1],
\]
defines an isomorphism $\det(P) \cong \det(P_1) \otimes \det(P_2)$.
\end{proof}

\begin{lemma} \label{lem Spinc prod}
Let $G$ be a Lie group, acting on a smooth manifold $N$. Let $H < G$ be a closed subgroup, and consider the fibred product $M := G \times_H N$. Let $E^N \to N$ be an oriented $H$-vector bundle of rank $r$, equipped with an $H$-invariant metric. Then, as in Subsection \ref{sec preq}, we can form the $G$-vector bundle
\[
E^M := G \times_H E^N \to M.
\]

If $P^N \to N$ is an $H$-equivariant $\Spin^c$-structure on $E$, then $P^M := G \times_H P^N$ is a $G$-invariant $\Spin^c$-structure on $E^M$. If $L^{N} \to N$ is the determinant line bundle of $P^N$, then the determinant line bundle of $P^M$ is $G \times_H L^{N}$.
\end{lemma}
\begin{proof}
The first claim is a direct consequence of the fact that the actions of $H$ and $\Spin^c(r)$ on $P^N$ commute. For the same reason, we have
\[
\begin{split}
\det(P^M) &= \bigl(G \times_H P^N \bigr) \times_{\Spin^c(r)} \C \\
	&= G \times_H \bigl(P^N \times_{\Spin^c(r)} \C \bigr) \\
	&= G \times_H L^{N}.
\end{split}
\]
\end{proof}

\subsubsection*{An induced $\Spin^c$-structure}

Let a $K$-equivariant $\Spin^c$-structure $P^N$ on $N$ be given. To construct a $G$-equivariant $\Spin^c$-structure on $M = G \times_K N$, we recall that, by Corollary \ref{cor TM},
\begin{equation}\label{eq deco TM}
TM \cong (p_{G/K}^* T(G/K)) \oplus (G \times_K TN),
\end{equation}
with $p_{G/K}: M \to G/K$ the natural projection. As in Subsection \ref{Laff}, we assume that the homomorphism $\Ad: K \to \SO(\p)$ lifts to a homomorphism $\widetilde{\Ad}: K \to \Spin(\p)$. Then $G/K$ carries the natural $\Spin$-structure
\[
P^{G/K} := G \times_K \Spin(\p),
\]
where $K$ acts on $\Spin(\p)$ via $\widetilde{\Ad}$.
\begin{lemma}
The principal $\Spin^c(\p)$-bundle
\[
P^{G/K}_M := G \times_K(N \times \Spin^c(\p)) \to M
\]
defines a $\Spin^c$-structure on $p_{G/K}^* T(G/K)$. Its determinant line bundle is trivial.
\end{lemma}
\begin{proof}
We have
\[
\begin{split}
G \times_K(N \times \Spin^c(\p)) \times_{\Spin^c(\p)} \p &\cong G \times_K(N \times \p) \\
	&\cong p_{G/K}^*(G \times_K \p) \\
	&\cong p_{G/K}^* T(G/K).
\end{split}
\]

Note that the determinant homomorphism is trivial on the subgroup $\Spin(\p) < \Spin^c(\p)$, and that $\widetilde{\Ad}(K) < \Spin(\p)$. Therefore, the action of $K$ on $\C$, given by the composition
\[
K \xrightarrow{\widetilde{\Ad}} \Spin(\p) \hookrightarrow \Spin^c(\p) \xrightarrow{\det} \U(1),
\]
is trivial. We conclude that
\[
\det\bigl(P^{G/K}_M \bigr) \cong G \times_K(N \times \C) \cong M \times \C,
\]
as claimed.
\end{proof}

Using the decomposition \eqref{eq deco TM} of $TM$, and the constructions from Lemmas \ref{lem Spinc sum} and \ref{lem Spinc prod}, we now obtain a $\Spin^c$-structure $P^M \to M$ on $M$, from the $\Spin^c$-structures $P^{G/K}_M \to M$ and $P^N \to N$. Explicitly,
\[
P^M := \bigl(G \times_K(N \times \Spin^c(\p)) \bigr) \times_{\U(1)} \bigl(G \times_K P^N\bigr) \times_H \Spin^c(d_M).
\]

By Lemmas  \ref{lem Spinc sum} and \ref{lem Spinc prod}, and by triviality of $\det\bigl(P^{G/K}_M \bigr)$, we see that the determinant line bundle of $P^M$ equals
\[
\det\bigl(P^M\bigr) = G \times_K \det\bigl(P^N\bigr).
\]
In particular, if the determinant line bundle of $P^N$ is a $\Spin^c$-prequantum line bundle $L^{2\nu} \to N$, then
\begin{equation} \label{eq det PM}
\det\bigl(P^M\bigr) = G \times_K L^{2\nu} = {L^{\omega}}
\end{equation}
is the $\Spin^c$-prequantum line bundle on $M$ constructed in Subsection \ref{sec preq}.

\section{Quantisation commutes with induction} \label{[Q,I]=0}

Our proof that quantisation commutes with reduction for semisimple groups is a reduction to the case of compact groups.
This reduction is possible because of the `quantisation commutes with \emph{induction}' result in this section
(Theorem \ref{thm [Q,I]=0}). It is analogous to Theorem 7.5 from \cite{Paradan1}. After stating this result, we show how,
 together with the quantisation commutes with reduction result for the compact case, it implies Theorem \ref{GSss}.
 Our proof that quantisation commutes with induction is based on  naturality of the assembly map for the inclusion $K \hookrightarrow G$ (Theorem \ref{natKG}).
 This proof is outlined in Subsection \ref{sec outline}, with details given in Sections \ref{natincl} and \ref{Dirac}.

\subsection{The sets $\CHPS(G)$ and $\CHPS(K)$}

We first restate the results of Sections \ref{cpt <-> ncpt} and \ref{sec preqSpin}  in a way that will allow us to draw a
`quantisation commutes with induction' diagram.
\begin{definition} \label{def HamP}
The set $\HP(G)$\index{ha@$\HP(G)$} of \emph{Ham}iltonian $G$-actions with momentum map values in the \emph{s}trongly \emph{e}lliptic set, with $\Spin^c$-\emph{p}requantisations, consists of classes of sextuples $(M, \omega, \Phi^M, {L^{2\omega}}, (\relbar, \relbar)_{{L^{2\omega}}}, \nabla^M)$, where
\begin{itemize}
\item $(M, \omega)$ is a symplectic manifold, equipped with a symplectic $G$-action;
\item $\Phi^M: M \to \g^*$ is a momentum map for this action, and $\Phi^M(M) \subset \gse$;
\item $\bigl({L^{2\omega}}, (\relbar, \relbar)_{{L^{2\omega}}}. \nabla^M \bigr)$ is a $G$-equivariant $\Spin^c$-quantisation of $(M, \omega)$.
\end{itemize}
Two classes $[M, \omega, \Phi^M, {L^{2\omega}}, (\relbar, \relbar)_{{L^{2\omega}}}, \nabla^M]$ and $[M', \omega', \Phi^{M'}, L^{2\omega'}, (\relbar, \relbar)_{L^{2\omega'}}, \nabla^{M'}]$
of such sextuples are identified
if there is an equivariant symplectomorphism $\varphi: M \to M'$ such that $\varphi^*\Phi^{M'} = \Phi^M$,  $\varphi^*L^{2\omega'} = {L^{2\omega}}$ and
$\varphi^*(\relbar, \relbar)_{L^{2\omega'}} = (\relbar, \relbar)_{{L^{2\omega}}}$. We do not require $\varphi$ to relate the connections $\nabla^M$ and $\nabla^{M'}$ to each other. For the purpose of quantisation, it is enough that it relates their curvatures by $\varphi^* R_{\nabla^{M'}} = R_{\nabla^M}$, which follows from the facts that $\varphi$ is a symplectomorphism, and that $\nabla^M$ and $\nabla^{M'}$ are prequantum connections.

Analogously, $\HP(K)$ is the set of classes $[N, \nu, \Phi^N, {L^{2\nu}}, (\relbar, \relbar)_{{L^{2\nu}}}, \nabla^N]$, where $(N, \nu)$ is a Hamiltonian $K$-manifold, with momentum map $\Phi^N$, with image in $\kse$, and $({L^{2\nu}}, (\relbar, \relbar)_{{L^{2\nu}}}, \nabla^N)$ is a $K$-equivariant $\Spin^c$-prequantisation of $(N, \nu)$. The equivalence relation between these classes is the same as before.
\end{definition}

Using this definition, we can summarise the results of Subsections \ref{sec ind}, \ref{sec res}, \ref{sec inv} and \ref{sec preq} as follows:
\begin{theorem}
There are well-defined maps
\[
\HInd_K^G: \HP(K) \to \HP(G)
\]
and\index{haa@$\HCross_K^G$}
\[
\HCross^G_K: \HP(G) \to \HP(K),
\]
given by
\[
\HInd_K^G[N, \nu, \Phi^N, {L^{2\nu}}, (\relbar, \relbar)_{{L^{2\nu}}}, \nabla^N] = [M, \omega, \Phi^M, {L^{2\omega}}, (\relbar, \relbar)_{{L^{2\omega}}},  \nabla^M]
\]
as in Subsections \ref{sec ind} and \ref{sec preq}, and
\[
\HCross^G_K[M, \omega, \Phi^M, {L^{2\omega}}, (\relbar, \relbar)_{{L^{2\omega}}},  \nabla^M] = [N, \nu, \Phi^N, {L^{2\nu}}, (\relbar, \relbar)_{{L^{2\nu}}}, \nabla^N]
\]
as in Subsections \ref{sec res} and \ref{sec preq}. They are each other's inverses.
\end{theorem}

To state our `quantisation commutes with reduction' result,
we need slightly different sets from $\HP(G)$ and $\HP(K)$. For these sets we only have an induction map,
and we do not know if it is possible to define a suitable cross-section map.
\begin{definition} \label{def CHPS}
The set $\CHPS(G)$\index{caad@$\CHPS$} of \emph{c}ocompact \emph{Ham}iltonian $G$-actions on complete manifolds, with momentum map values in the \emph{s}trongly \emph{e}lliptic set,
 with $\Spin^c$-\emph{p}re\-quan\-ti\-sa\-tions and $\Spin^c$-structures, consists of classes of septuples
$(M, \omega, \Phi^M, {L^{2\omega}}, (\relbar, \relbar)_{{L^{2\omega}}}, \nabla^M, P^M)$,
 with $(M, \omega, \Phi^M, {L^{2\omega}}, (\relbar, \relbar)_{{L^{2\omega}}}, \nabla^M)$ as in Definition \ref{def HamP}, $M/G$ compact, and $P^M$ a $G$-equivariant  $\Spin^c$-structure on $M$, such that
 \begin{itemize}
 \item $M$ is complete in the Riemannian metric induced by $P^M$;
 \item the determinant line bundle of $P^M$ is isomorphic to $L^{2\omega}$.
 \end{itemize}

The equivalence relation is the same as in Definition \ref{def HamP}. There is no need to incorporate the $\Spin^c$-structures into this equivalence relation, besides the condition on the determinant line bundles of these structures that is already present.

The set $\CHPS(K)$ is defined analogously. In this case, the condition that $N/K$ is compact is equivalent to compactness of $N$.
\end{definition}
For these sets, we have the induction map
\begin{equation} \label{eq HInd}
\HInd_K^G: \CHPS(K) \to \CHPS(G),
\end{equation}
with \index{haa@$\HInd_K^G$}
\[
\HInd_K^G[N, \nu, \Phi^N, {L^{2\nu}}, (\relbar, \relbar)_{{L^{2\nu}}}, \nabla^N, P^N] =
	[M, \omega, \Phi^M, {L^{2\omega}}, (\relbar, \relbar)_{{L^{2\omega}}}, \nabla^M, P^M],
\]
as defined in Subsections \ref{sec ind}, \ref{sec preq} and \ref{sec Spinc}.

\subsection{Quantisation commutes with induction}

Consider an element $[M, \omega, \Phi^M, {L^{2\omega}}, (\relbar, \relbar)_{{L^{2\omega}}}, \nabla^M, P^M] \in \CHPS(G)$.
Using a connection on the spinor bundle associated to $P^M$, we can define the $\Spin^c$-Dirac operator $\D_M^{{L^{2\omega}}}$ on $M$, as in Subsection \ref{sec quant Spinc}. In Definition \ref{def quant VI}, we defined the quantisation of the action of $G$ on $(M, \omega)$ as the image of the $K$-homology class of $\D_{M}^{{L^{2\omega}}}$ under the analytic assembly map:
\[
Q_{\Spin}^G(M, \omega) = \mu_M^G\bigl[\D_{M}^{{L^{2\omega}}} \bigr].
\]
as we noted before, this definition does not depend on the choice of connection on the spinor bundle.
\begin{definition}
The quantisation map\index{qaah@$Q_G$, $Q_K$}
\[
Q_{\Spin}^G: \CHPS(G) \to K_0(C^*_r(G))
\]
is defined by
\[
Q_{\Spin}^G[M, \omega, \Phi^M, {L^{2\omega}}, (\relbar, \relbar)_{{L^{2\omega}}}, \nabla^M, P^M] = \mu_M^G\bigl[{\D}_M^{{L^{2\omega}}} \bigr].
\]

Analogously, we have the quantisation map
\[
Q_{\Spin}^K: \CHPS(K) \to K_0(C^*_rK)
\]
given by
\[
Q_{\Spin}^K[N, \nu, \Phi^N, {L^{2\nu}}, (\relbar, \relbar)_{{L^{2\nu}}}, \nabla^N, P^N] = \mu_N^K\bigl[{\D}_N^{{L^{2\nu}}} \bigr],
\]
which corresponds to $\Kind \, {\D}_N^{{L^{2\nu}}} \in R(K)$.
\end{definition}

Using the Dirac induction map \eqref{eq DInd} and the Hamiltonian induction map \eqref{eq HInd}, we can now state the following result:
\begin{theorem}[Quantisation commutes with induction] \label{thm [Q,I]=0} \index{quantisation commutes with induction}

\noindent
The following diagram commutes:
\begin{equation} \label{diag [Q,I]}
\xymatrix{
\CHPS(G) \ar[r]^-{Q_{\Spin}^G} & K_0(C^*_r(G)) \\
\CHPS(K) \ar[r]^-{Q_{\Spin}^K} \ar[u]^{\HInd_K^G} & R(K). \ar[u]_{\DInd_K^G}
}
\end{equation}
\end{theorem}
This is the central result of this paper. We will outline its proof in Subsection \ref{sec outline}, and fill in the details in Sections \ref{natincl} and  \ref{Dirac}.

\subsection{Corollary: $[Q,R]=0$ for semisimple groups} \label{sec [Q,I] = [Q,R]}

As announced, we derive Theorem \ref{GSss} from Theorem \ref{thm [Q,I]=0} and the fact that $\Spin^c$-quantisation commutes with reduction in the compact case (Theorem \ref{thm [Q,R]=0 cpt spin}).

\medskip
\noindent \emph{Proof of Theorem \ref{GSss}.}
Let $G$, $K$, $(M, \omega)$, $\Phi^M=\Phi$, ${L^{2\omega}}=L$, $(\relbar, \relbar)_{{L^{2\omega}}}=(\relbar, \relbar)_{L}$ and $\nabla^M=\nabla$  be as in Theorem \ref{GSss}.
Set
\[
(N, \nu, \Phi^N, {L^{2\nu}}, (\relbar, \relbar)_{{L^{2\nu}}}, \nabla^N) := \HCross^G_K(M, \omega, \Phi^M, {L^{2\omega}}, (\relbar, \relbar)_{{L^{2\omega}}}, \nabla^M).
\]
Let $P^N \to N$ be a $K$-equivariant $\Spin^c$-structure on $N$, with determinant line bundle ${L^{2\nu}}$. Let $P^M \to M$ be the induced $\Spin^c$-structure on $M$, as described in Subsection \ref{sec Spinc}. Since the determinant line bundle of $P^M$ is ${L^{2\omega}}$, by \eqref{eq det PM} and part \emph{(ii)} of Lemma \ref{lem lbindres}, we have the elements
\[
\begin{split}
[N, \nu, \Phi^N, {L^{2\nu}}, (\relbar, \relbar)_{{L^{2\nu}}}, \nabla^N, P^N] &\in \CHPS(K); \\
[M, \omega, \Phi^M, {L^{2\omega}}, (\relbar, \relbar)_{{L^{2\omega}}}, \nabla^M, P^M] &\in \CHPS(G).
\end{split}
\]
By Proposition \ref{prop Ind Res inv}, we have
\[
\HInd_K^G[N, \nu, \Phi^N, {L^{2\nu}}, (\relbar, \relbar)_{{L^{2\nu}}}, \nabla^N, P^N] =
	[M, \omega, \Phi^M, {L^{2\omega}}, (\relbar, \relbar)_{{L^{2\omega}}}, \nabla^M, P^M].
\]

Now let $\HH$ and $\lambda$ be as in Theorem \ref{GSss}. Then by Theorem \ref{thm [Q,I]=0} and Lemma \ref{lem Laff}, and the fact that the assembly map is the regular index in teh compact case, we have
\[
\begin{split}
R^{\HH}_G \circ \mu_M^G \bigl[{\D}_M^{{L^{2\omega}}} \bigr] &= R^{\HH}_G \circ \DInd_K^G (\Kind \, {\D}_N^{{L^{2\nu}}}) \\
	&= (-1)^{\dim G/K} [\Kind \, {\D}_N^{{L^{2\nu}}} : V_{\lambda-\rho_c}].
\end{split}
\]
Because $\Spin^c$-quantisation commutes with reduction for the action of $K$ on $N$ (Theorem \ref{thm [Q,R]=0 cpt spin}), we have
\[
[\Kind \, {\D}_N^{{L^{2\nu}}} : V_{\lambda-\rho_c}] = Q_{\Spin}\bigl(N_{\lambda}, \omega_{\lambda}\bigr)
\]
if $-i\lambda \in \Phi^N(N)$, and zero otherwise.
Recall that $N = \bigl(\Phi^M\bigr)^{-1}(\kk^*)$, so that $-i\lambda \in \Phi^N(N)$ if and only if $-i\lambda \in \Phi^M(M)$.
Furthermore, note that $G_{\nu} \subset K$ for all $\nu \in \torus^*_{+} \setminus \ncw$, so that $G_{\nu}=K_{\nu}$ for such $\nu$.
Therefore $N_{\lambda} = M_{\lambda}$,
which completes the proof.
\hfill $\square$

\subsection{Outline of the proof} \label{sec outline}

The most important ingredient of the proof of Theorem \ref{thm [Q,I]=0} is Theorem \ref{natKG}, `naturality of the assembly map for the inclusion of $K$ into $G$'.

As before, let $K < G$ be a maximal compact subgroup. Let $N$ be a smooth manifold, equipped with a  $K$-action. Let $M := G \times_K N$ be the quotient of $G \times N$ by the $K$-action given by
\[
k\cdot (g,n) = (gk^{-1}, kn),
\]
for $k \in K$, $g \in G$ and $n \in N$. Because this action is proper and free, $M$ is a smooth manifold. Left multiplication on the factor $G$ induces an action of $G$ on $M$.
\begin{theorem}[Naturality of the assembly map for $K \hookrightarrow G$] \label{natKG} \index{naturality of the assembly map!for $K \hookrightarrow G$}
The map $\KInd_K^G$, defined by commutativity of the left hand side of diagram \eqref{diagram}, makes the following diagram commutative:
\begin{equation} \label{diag natKG}
\xymatrix{K_0^G(M) \ar[r]^{\mu_M^G} & K_0(C^*_r(G)) \\
K_0^K(N) \ar[r]^{\mu_N^K} \ar[u]^{\KInd_K^G} & R(K). \ar[u]_{\DInd_K^G}
}
\end{equation}
\end{theorem}
This result is analogous to Theorem 4.1 from \cite{Atiyah}, which is used by Paradan in \cite{Paradan1} to reduce
the Guillemin--Sternberg conjecture for compact groups to certain subgroups. Our proof of Theorem \ref{GSss} is analogous to this part of Paradan's work.

The reason why Theorem \ref{natKG} helps us to prove Theorem \ref{thm [Q,I]=0} is the fact that the map
$\KInd_K^G$ that appears in Theorem \ref{natKG} relates the Dirac operators ${\D}_N^{{L^{2\nu}}}$ and ${\D}_M^{{L^{2\omega}}}$ to each other:
\begin{proposition} \label{alphaDirac}
The map $\KInd_K^G$ maps the $K$-homology class of the operator ${\D}_N^{{L^{2\nu}}}$ to the class of ${\D}_M^{{L^{2\omega}}}$.
\end{proposition}

Combining Theorem \ref{natKG} and Proposition \ref{alphaDirac}, we obtain a proof of Theorem \ref{thm [Q,I]=0}:

\medskip
\noindent \emph{Proof of Theorem \ref{thm [Q,I]=0}.}
Let
\[
x = [N, \nu, \Phi^N, {L^{2\nu}}, (\relbar, \relbar)_{{L^{2\nu}}}, \nabla^N, P^N] \in \CHPS(K)
\]
be given, and write
\[
[M, \omega, \Phi^M, {L^{2\omega}}, (\relbar, \relbar)_{{L^{2\omega}}}, \nabla^M, P^M] := \HInd_K^G(x).
\]
Then by Proposition \ref{alphaDirac} and Theorem \ref{natKG},
\[
\begin{split}
Q_{\Spin}^G\bigl(\HInd_K^G (x)\bigr) &= \mu_M^G \bigl[{\D}_M^{{L^{2\omega}}} \bigr] \\
	&= \mu_M^G \circ \KInd_K^G \bigl[{\D}_N^{{L^{2\nu}}} \bigr] \\
	&= \DInd_K^G \circ \mu_N^K \bigl[{\D}_N^{{L^{2\nu}}} \bigr] \\
	&= \DInd_K^G \bigl( Q_{\Spin}^K(x) \bigr).
\end{split}
\]
\hfill $\square$

\medskip \noindent
It remains to prove Theorem \ref{natKG} and Proposition \ref{alphaDirac}. These proofs will be given in Sections \ref{natincl} and \ref{Dirac}.

\section{Naturality of the assembly map} \label{natincl}

We will prove Theorem \ref{natKG} by decomposing diagram \eqref{diag natKG} as follows:
\begin{equation}\label{diagram}
\xymatrix{
K_0^{G}(M) \ar[r]^{\mu_M^G} & K_0(C^*_r(G)) \\
K_0^{G \times \Delta(K)}(G \times N) \ar[r]^{\mu_{G \times N}^{G \times \Delta(K)}} \ar[u]_{V_{\Delta(K)}} & K_0(C^*_r(G \times K)) \ar[u]^{R^0_{K}} \\
K_0^{G \times K \times K}(G \times N) \ar[r]^{\mu_{G \times N}^{G \times K \times K}} \ar[u]^{\Res^{G \times K \times K}_{G \times \Delta(K)}}
		& K_0(C^*_r(G \times K \times K)) \ar[u]_{\Res^{G \times K \times K}_{G \times \Delta(K)}} \\
K_0^K(N) \ar[r]_{\mu_N^K} \ar@/^5pc/[uuu]^{\KInd_K^G} \ar[u]_{[{\D}_{G,K}] \times \relbar}
		& R(K).  \ar@/_5pc/[uuu]_{\DInd_K^G}  \ar[u]^{\mu_{G}^{G \times K}[{\D}_{G,K}] \times \relbar}
}
\end{equation}
In this diagram, all the horizontal maps involving the letter $\mu$ are analytic assembly maps. The symbol `$\times$' denotes the Kasparov product, and
$\Delta(K)$ is the diagonal subgroup of $K \times K$. The map $\DInd_K^G$ was defined in \eqref{IGK}.
The other maps will be defined in the remainder of this chapter.

The $K$-homology class $[{\D}_{G,K}] \in K_0^{G \times K}(G)$ is defined as follows.
Note that the $\Spin$-Dirac operator on $G/K$ is the operator ${\D}_{G/K} = \D^{\C}$, with $\C$ the trivial $K$-representation,
and $\D^{\C}$ as in \eqref{DV}. Let $p_G: G \to G/K$ be the
quotient map, let $\SSS^{G/K} := G \times_K \Delta_{\p}$ be the spinor bundle on $G/K$, and consider the trivial vector bundle $p_G^*\SSS^{G/K} = G \times \Delta_{d_{\p}} \to G$.
Let ${\D}_{G,K}$\index{daaaaa@${\D}_{G,K}$} be the operator on this bundle given by the same formula \eqref{DV} as the operator ${\D}^V$, with $V=\C$ the trivial representation.
This operator satisfies
\[ \label{DGK}
{\D}_{G,K}(p_G^*s) = p_G^* \bigl(\D^{\C}s \bigr),
\]
for all sections $s$ of $\SSS^{G/K} \to G/K$.
We will use the fact that it is equivariant with respect to the action of $G \times K$ on $G \times \Delta_{d_{\p}}$ defined by
\[
(g,k) \cdot (g', \delta) = (gg'k^{-1}, \widetilde{\Ad}(k) \cdot \delta),
\]
for $g, g' \in G$, $k \in K$ and $\delta \in \Delta_{d_{\p}}$. It is elliptic (see Lemma \ref{lem symb DGK}), and therefore defines a class $[{\D}_{G,K}] \in K_0^{G \times K}(G)$.

We will distinguish between the different subdiagrams of \eqref{diagram} by calling them the `left-hand', `top', `middle', `bottom' and `right-hand' diagrams.
Commutativity of the left-hand diagram is the definition of the map $\KInd_K^G$.\index{kaaf@$\KInd_K^G$} In this chapter we will prove that the other diagrams commute as well, thus giving a proof of Theorem \ref{natKG}.

\subsection[Naturality of the assembly map for epimorphisms]{The top diagram: naturality of the assembly map for epimorphisms} \label{sec top}

In this subsection, we suppose that $G$ is a locally compact Hausdorff group, and that $K \lhd G$ is a compact \emph{normal} subgroup of $G$. Furthermore,
let $X$ be a locally compact, Hausdorff, proper $G$-space such that $X/G$ is compact. Commutativity of the the top diagram is a special case of
commutativity of the following diagram:
\begin{equation} \label{diagnat}
\xymatrix{
K_0^{G/K}(X/K) \ar[r]_-{\mu_{X/K}^{G/K}} \ar@/^1.6pc/[rr]^{\mu_{X/K}^{G/K}} & K_0(C^*(G/K)) \ar[r]_-{\lambda_{G/K}} & K_0(C^*_r(G/K)) \\
K_0^G(X) \ar[r]^-{\mu_X^G} \ar[u]^{V_K}  \ar@/_1.6pc/[rr]_{\mu_X^G}& K_0(C^*(G)) \ar[r]^-{\lambda_G} \ar[u]_{R^0_K} & K_0(C^*_r(G)). \ar[u]_{R^0_K}
}
\end{equation}
We have used the same notation for the assembly map with respect to the full group $C^*$-algebra as for the assembly map with respect to the reduced one.

The maps $\lambda_{G/K}$ and $\lambda_G$ are by definition induced by the maps
\[
 \begin{split}
  C^*G &\to C^*_rG; \\
	C^*(G/K) &\to C^*_r(G/K),
 \end{split}
\]
defined by continuously extending the identity maps on $C_c(G)$ and $C_c(G/K)$, respectively. It is not hard to check that the right hand diagram in \eqref{diagnat} commutes.

Commutativity of the left hand diagram in \eqref{diagnat} is a special case of naturality of the assembly map for epimorphisms. This is proved in Valette's part of \cite{MV} for discrete groups. In \cite{HL}, it is indicated how to generalise this result to possibly nondiscrete groups. The notation `$V_K$' for the map in the left hand diagram is used in \cite{HL}.

It is a striking feature of our version of naturality of the assembly map for the monomorphism $K \hookrightarrow G$  that it actually relies on the epimorphism case in this way.

\subsection[Restriction to subgroups]{The middle diagram: restriction to subgroups}

In the middle diagram of \eqref{diagram}, the map
\[
\Res^{G \times K \times K}_{G \times \Delta(K)}: K_0^{G \times K \times K}(G \times N) \to K_0^{G \times \Delta(K)}(G \times N)
\]
is simply given by restricting representations and actions of $G \times K \times K$  to $G \times \Delta(K)$. The other restriction map,
\begin{equation} \label{resKth}
\Res^{G \times K \times K}_{G \times \Delta(K)}: K_0(C^*_r(G\times K \times K)) \to K_0(C^*_r(G \times \Delta(K))),
\end{equation}
is harder to define. (The restriction
map $C_c(G \times K \times K) \to C_c(G \times \Delta(K))$ is not continuous in the norms of the reduced group $C^*$-algebras involved, for example.)

 We define the map \eqref{resKth} using the \emph{K\"unneth formula}. Since $G$ is a connected Lie group (in particular, it is an almost connected locally compact topological group),
 it satisfies the Baum--Connes conjecture with arbitrary $G$-trivial coefficients (see \cite{CEO}, Corollary 0.5).
 By Corollary 0.2 of \cite{CEO}, the algebra $C^*_r(G)$ therefore satisfies the K\"unneth formula. In particular,
\[
\begin{split}
K_0(C^*_r(G \times K \times K)) &\cong K_0(C^*_r(G) \otimes_{\mathrm{min}} C^*_r(K \times K)) \\
	&\cong K_0(C^*_r(G)) \otimes K_0(C^*_r(K \times K)) \\
	&\cong K_0(C^*_r(G)) \otimes R(K\times K).
\end{split}
\]
 Here we have used the fact that the representation ring $R(K \times K)$ is torsion-free, and the fact that $C^*_r(G_1) \otimes_{\mathrm{min}} C^*_r(G_2) \cong C^*_r(G_1 \otimes G_2)$ for all locally compact Hausdorff groups $G_1$ and $G_2$. Analogously, we have an isomorphism $K_0(C^*_r(G \times K))\cong K_0(C^*_r(G)) \otimes R(K)$.

The isomorphism is given by the Kasparov product. This product is defined as the composition
\begin{multline} \label{eq Kasprod}
\KK_0(\C, C^*_r(G)) \otimes \KK_0(\C, C^*_r(K \times K)) \xrightarrow{1 \otimes \tau_{C^*_r(G)}} \\
 \KK_0(\C, C^*_r(G)) \otimes \KK_0(C^*_r(G), C^*_r(G) \otimes_{\mathrm{min}} C^*_r(K \times K)) \xrightarrow{\times_{C^*_r(G)}} \\
\KK_0(\C, C^*_r(G) \otimes_{\mathrm{min}} C^*_r(K \times K)),
\end{multline}
where $\tau_{C^*_r(G)}$ is defined by tensoring from the left by $C^*_r(G)$, and $\times$ denotes the Kasparov product (see \cite{Blackadar}, Chapter 18.9).  Let
\[
\Res^{K \times K}_{\Delta(K)}: R(K \times K) \to R(\Delta(K)) = R(K)
\]
be the usual restriction map to the diagonal subgroup. We define \eqref{resKth} as the map
\[
1_{K_0(C^*_r(G))} \otimes \Res^{K \times K}_{\Delta(K)}:  K_0(C^*_r(G)) \otimes R(K\times K) \to  K_0(C^*_r(G)) \otimes R(K).
\]

Commutativity of the middle diagram now follows from
\begin{lemma} \label{lem Res}
Let $X$ be a locally compact,  Hausdorff, proper $G\times K$-space with compact quotient, and let $Y$ be a compact,  Hausdorff $K$-space. Then the following diagram commutes:
\[
\xymatrix{
K_0^{G \times \Delta(K)}(X\times Y) \ar[r]^{\mu_{X \times Y}^{G \times \Delta(K)}} & K_0(C^*_r(G \times K)) \\
K_0^{G\times K\times K}(X\times Y) \ar[u]^{\Res^{G \times K \times K}_{G \times \Delta(K)}} \ar[r]^{\mu_{X\times Y}^{G\times K \times K}}
		& K_0(C^*_r(G \times K \times K)). \ar[u]_{\Res^{G \times K \times K}_{G \times \Delta(K)}}
}
\]
\end{lemma}
\begin{proof}
Let $a = [\HH, F, \pi] \in K_0^{G\times K\times K}(X\times Y)$, $b = [\E_G, F_G] \in K_0(C^*_r(G))$ and $[V] \in R(K\times K)$ be given, such that
\[
\mu_{X\times Y}^{G \times K \times K}(a) = b\times [C^*_r(G) \otimes V] = [\E_G \otimes V, F_G \otimes 1_V] \in K_0(C^*_r(G\times K \times K)).
\]
Because the assembly and restriction maps are $\Z$-module homomorphisms, it is sufficient to prove the claim in this case where the image of $a$ is a simple tensor.

If we write
\[
\begin{split}
[\E, F_{\E}] & := \mu_{X\times Y}^{G \times K \times K}(a) \quad \in K_0(C^*_r(G \times K \times K)); \\
[\E', F_{\E'}] & := \mu_{X \times Y}^{G \times \Delta(K)} \circ  \Res^{G \times K \times K}_{G \times \Delta(K)}(a) \quad \in K_0(C^*_r(G \times K)),
\end{split}
\]
then the operators $F_{\E}$ and $F_{\E'}$ coincide on the dense mutual subspace $\HH_c$ of $\E$ and $\E'$. It is therefore enough to prove that
\[
\E' \cong \E_G \otimes_{\C} \bigl(V|_{\Delta(K)} \bigr)
\]
as Hilbert $C^*_r(G\times K)$-modules.

Using the usual choice of representatives of the classes $b$ and $[\E, F_{\E}]$  we have an isomorphism of Hilbert $C^*_r(G \times K \times K)$-modules
\[
\psi: \E \xrightarrow{\cong} \E_G \otimes V.
\]
Define the map
\[
\varphi: \E' \xrightarrow{\cong} \E_G \otimes \bigl( V|_{\Delta(K)} \bigr)
\]
by $\varphi|_{\HH_c} = \psi|_{\HH_c}$, and continuous extension. The map $\varphi$ is well-defined, and indeed an isomorphism, if it is a homomorphism of Hilbert $C^*_r(G\times K)$-modules. To show that $\varphi$ preserves the $C^*_r(G \times K)$-valued inner products, let $\xi_1, \xi_2 \in \HH_c$ be given, and suppose that $\varphi(\xi_j) = e_j \otimes v_j \in \E_G \otimes V$ for $j=1,2$. (By linearity of $\varphi$, it is indeed enough to consider the case where the $\varphi(\xi_j)$ are simple tensors.) Then for all $g \in G$ and $k \in K$,
\[
\begin{split}
\bigl( \varphi(\xi_1), \varphi(\xi_2) \bigr)_{\E_G \otimes V|_{\Delta(K)}} (g,k) &= ( e_1, e_2)_{\E_G} (g) \bigl(v_1, (k, k)\cdot v_2\bigr)_V \\
	&= \bigl( \psi(\xi_1), \psi(\xi_2)\bigr)_{\E_G \otimes V} (g,k,k) \\
	&= ( \xi_1, \xi_2)_{\E} (g,k,k),
\end{split}
\]
because $\psi$ is an isomorphism of Hilbert $C^*(G \times K \times K)$-modules. The latter expression equals
\[
\bigl(\xi_1, (g,k,k)\cdot \xi_2\bigr)_{\HH} = (\xi_1, \xi_2)_{\E'}(g,k),
\]
which shows that $\varphi$ preserves the inner products.

Finally, because $\psi$ is a homomorphism of $C^*_r(G \times K \times K)$-modules, the map $\varphi$ is a homomorphism of $C^*_r(G \times K)$-modules on $\HH_c$, and hence on all of $\E'$.
\end{proof}

\subsection[Multiplicativity of the assembly map]{The bottom diagram: multiplicativity of the assembly map} \label{sec bottom}

Commutativity of the bottom diagram is a special case of the multiplicativity property of the assembly map that we will prove in this subsection.
This property generalises multiplicativity of the index with respect to Atiyah's `sharp product' of elliptic operators, as described in \cite{Atiyah}, Theorem 3.5. In this subsection, we will denote the tensor product of Hilbert $C^*$-modules by $\hat \otimes$, to emphasise the difference with the algebraic tensor product $\otimes$.

For this subsection, let $G_1$ and ${G_2}$ be locally compact Hausdorff topological groups, acting properly on two locally compact metrisable spaces ${X_1}$ and ${X_2}$, respectively. Suppose ${X_1}/G_1$ and ${X_2}/{G_2}$ are compact. Consider the Kasparov product maps
\begin{align}
K_0^{G_1}(X_1) \otimes K_0^{G_2}(X_2) & \xrightarrow{\times} K_0^{G_1 \times G_2}(X_1 \times X_2); \nonumber \\
K_0(C^*_{(r)}(G_1)) \otimes  K_0(C^*_{(r)}(G_2)) & \xrightarrow{\times} K_0(C^*_{(r)}(G_1 \times G_2)). \label{KasproductGH}
\end{align}
Here the symbol $C^*_{(r)}$ denotes either the full or the reduced group $C^*$-algebra, and we have used the $C^*$-algebra isomorphisms
\[
 C_0(X_1) \otimes C_0(X_2) \cong C_0(X_1 \times X_2)
\]
for all locally compact Hausdorff spaces $X_1$ and $X_2$,
and
\[
\begin{split}
C^*(G_1) \otimes_{\mathrm{max}} C^*(G_2) &\cong C^*(G_1 \times G_2); \\
C^*_r(G_1) \otimes_{\mathrm{min}} C^*_r(G_2) &\cong C^*_r(G_1 \times G_2).
\end{split}
\]
for locally compact Hausdorff groups $G_1$ and $G_2$.

Analogously to \eqref{eq Kasprod}, the Kasparov product \eqref{KasproductGH} is actually the composition
\begin{multline}\label{KasprodKth}
 \KK_0(\C, C^*_{(r)}(G_1))  \otimes \KK_0(\C, C^*_{(r)}{(G_2)}) \xrightarrow{1\otimes \tau_{C^*_{(r)}{(G_1)}}} \\
 \KK_0(\C, C^*_{(r)}{(G_1)})  \otimes \KK_0(C^*_{(r)}{(G_1)}, C^*_{(r)}{(G_1)} \otimes C^*_{(r)}{(G_2)}) \xrightarrow{\times_{C^*_{(r)}(G_1)}} \\
 \KK_0(\C, C^*_{(r)}{(G_1)}  \otimes C^*_{(r)}{(G_2)}) = \KK_0(\C, C^*_{(r)}({G_1} \times {G_2})).
\end{multline}
The tensor product denotes the maximal tensor product in the case of full $C^*$-algebras, and the minimal tensor product for reduced $C^*$-algebras.

\begin{theorem}[Multiplicativity of the assembly map] \label{multass} \index{analytic assembly map@(analytic) assembly map!multiplicativity of --}
If ${X_1}$ and ${X_2}$ are metrisable, then for all $a_j \in K_0^{G_j}({X_j})$, we have
\[
\mu_{X_1}^{G_1}(a_1) \times \mu_{X_2}^{G_2}(a_2) = \mu_{{X_1} \times {X_2}}^{{G_1} \times {G_2}}(a_1 \times a_2) \quad \in K_0(C^*_{(r)}({G_1} \times {G_2})).
\]
\end{theorem}
Here the assembly maps are defined with respect to either the full of the reduced group $C^*$-algebras. We suppose $X_1$ and $X_2$ to be metrisable, because the $C^*$-algebras $C_0(X_1)$ and $C_0(X_2)$ are then separable, so that we can use Baaj and Julg's unbounded description of the Kasparov product. Theorem \ref{multass} may well be true for non-metrisable spaces, but we will only apply it to smooth manifolds anyway.

\subsubsection*{The assembly map in the unbounded picture}

In the proof of Theorem \ref{multass}, we will use the unbounded picture of $\KK$-theory (see \cite{BJ}), because of the easy form of the Kasparov product in this setting.
The construction of the unbounded assembly map given below works for full group $C^*$-algebras, so the following proof applies only to this case. Theorem \ref{multass} for reduced group $C^*$-algebras
can then be deduced using the maps $\lambda_{G_1}$ and $\lambda_{G_2}$ defined in Subsection \ref{sec top}.

For full group $C^*$-algebras, the assembly map in the unbounded picture is defined in Kucerovsky's appendix to \cite{MV}, in the following way. Let $G$ be a second countable, locally compact Hausdorff group, acting properly on a locally compact Hausdorff space $X$, with compact quotient. The assembly map in the unbounded picture is given by
\begin{equation} \label{eq unbass}
\mu_{X}^{G}(\HH, D, \pi) = (\E, D_{\E}) \quad \in \Psi_0(\C, C^*{G}),
\end{equation}
for all $(\HH, D, \pi ) \in \Psi_0^{G}(C_0({X}), \C)$. The Hilbert $C^*(G)$-module $\E$ is defined as usual for the
assembly map. The definition of the operator $D_{\E}$ on $\E$ is more involved.

First, let $\tilde \HH$ be the auxiliary Hilbert $C^*(G)$-module defined as the completion of the Hilbert $C_c(G)$-module $C_c(G, \HH)$ with respect to the $C_c(G) \subset C^*(G)$-valued inner product
\begin{equation} \label{eq inprod H tilde}
(\varphi, \psi)_{\tilde \HH}(g) := \int_G \bigl(\varphi(g'), \psi(g'g) \bigr)_{\HH}\, dg',
\end{equation}
where $\varphi, \psi \in C_c(G, \HH)$, $g \in G$, and $dg'$ is a Haar measure on $G$. Next, let $h \in C_c(X)$ be a function such that for all $x \in X$,
\[
\int_G h^2(gx)\, dg = 1
\]

Let $p \in C_c(X \times G)$ be the projection given by
\begin{equation} \label{eq p}
p(x,g) := \overline{h(x)}h(g^{-1}x).
\end{equation}
This function is compactly supported by properness of the action of $G$ on $X$. Let $\tilde \pi: C_c(X \times G) \to \B(\tilde \HH)$ be the representation given by
\[
\bigl(\tilde \pi(f)\varphi \bigr)(g) = \int_G \pi(f(\relbar, g')) g' \cdot \varphi(g'^{-1}g)\, dg',
\]
for $f \in C_c(X \times G)$, $\varphi \in C_c(G,\HH)$ and $g \in G$.
(The representation $\tilde \pi$ can actually be extended to the crossed product $C_0(X) \rtimes G$, but we will not use this extension.)

Then the map
\[
\alpha: \tilde \pi(p)C_c(G, \HH) \to \HH_c,
\]
given by
\[
\tilde \pi(p) \varphi \mapsto \int_G g^{-1}\pi(h) \varphi(g) \, dg,
\]
preserves the $C^*(G)$-valued inner products and the $C^*(G)$-module structures on $\tilde \HH$ and on $\E$, and induces an isomorphism $\tilde \pi(p) \tilde \HH \cong \E$ of Hilbert $C^*(G)$-modules. We will write $\tilde \E := \tilde \pi(p) \tilde \HH$.

To define the operator $D_{\E}$ on $\E$ we first consider an operator $D_{\tilde \E}$ on $\tilde \E$. This operator is defined as the closure of the operator $\tilde D$ on $\tilde \E$, given by
\begin{equation} \label{eq tilde D}
\tilde D \bigl(\tilde \pi(p) \varphi \bigr) := \tilde \pi(p) \bigl( D \circ \varphi \bigr),
\end{equation}
on the domain $\dom \tilde D := \tilde \pi(p) C_c(G, \dom D)$. We finally set
\[
D_{\E} := \alpha D_{\tilde \E} \alpha^{-1},
\]
on the domain $\dom D_{\E} = \alpha\bigl( \dom D_{\tilde \E}\bigr)$.

In the proof of Theorem \ref{multass}, we will actually use the following definition of the assembly map:
\begin{equation} \label{eq ass unb}
\tilde \mu_X^G(\HH,D,\pi) := \bigl(\tilde \E, D_{\tilde \E}\bigr) \quad \in \Psi_0(\C, C^*{G}),
\end{equation}
which gives the same class in $K_0(C^*(G))$ as \eqref{eq unbass}, because $\alpha$ is an isomorphism.

Kucerovsky's proof that the above constructions give a well-defined description of the assembly map in the unbounded picture is valid for discrete groups,
but it admits a straightforward generalisation to possibly nondiscrete (unimodular) ones.
One simply replaces sums by integrals, and uses the fact that the integral over a compact, finite Borel space of a continuous family of adjointable operators is again an adjointable operator.
In addition, in
the proof of Lemma 2.15 in \cite{MV}, one takes $\beta^{-1}(\pi(f) \eta) = \tilde \pi(p) \psi$, with $\psi(g) = \pi(h) \pi(g\cdot f) g\cdot \eta$ (where the $\beta$ in \cite{MV} is our $\alpha$).
This reduces to
Valette's $\beta^{-1}(\pi(f)\eta) = \tilde \pi(p)\tilde \pi(\langle h|f\rangle)  \bar \eta$  in the discrete case.

\subsubsection*{Proof of Theorem \ref{multass}}
For $j=1,2$, let
\[
a_j = (\HH_j, D_j, \pi_j) \in \Psi_0^{G_j}(C_0(X_j), \C)
\]
be given. Then
\[
\tilde \mu_{X_j}^{G_j}(a_j) = \bigl(\tilde \E_j, D_{\tilde \E_j}\bigr),
\]
as in \eqref{eq ass unb}. The product of $\tilde \mu_{X_1}^{G_1}(a_1)$ and $\tilde \mu_{X_2}^{G_2}(a_2)$ is
\begin{equation} \label{eq prod ass}
\tilde \mu_{X_1}^{G_1}(a_1) \times \tilde \mu_{X_2}^{G_2}(a_2) =
(\tilde \E_1 \hat \otimes \tilde \E_2, D_{\tilde \E_1 \hat \otimes \tilde \E_2}) \quad \in \Psi_0(\C, C^*(G_1 \times G_2)).
\end{equation}
Here $D_{\tilde \E_1 \hat \otimes \tilde \E_2}$ is the closure of the operator
\[
D_{\tilde \E_1} \otimes 1_{\tilde \E_2} + 1_{\tilde \E_2} \otimes D_{\tilde \E_2},
\]
on the domain $\dom D_{\tilde \E_1} \otimes \dom D_{\tilde \E_2}$.

On the other hand, the product $a_1 \times a_2$ is
\begin{equation} \label{eq a1 maal a2}
(\HH_1 \hat \otimes \HH_2, D_{\HH_1 \hat \otimes \HH_2}, \pi) \quad \in \Psi_0^{G_1 \times G_2}(C_0(X_1 \times X_2), \C),
\end{equation}
with $D_{\HH_1 \hat \otimes \HH_2}$ the closure of the operator
\[
D_1 \otimes 1_{\HH_2} + 1_{\HH_1} \otimes D_2
\]
on $\dom D_1 \otimes \dom D_2$. Furthermore,
 we have abbreviated $\pi := \pi_1 \otimes \pi_2$ for later convenience.
Applying the unbounded assembly map $\tilde \mu_{X_1 \times X_2}^{G_1 \times G_2}$ to the cycle \eqref{eq a1 maal a2}, we obtain
\begin{equation} \label{eq ass prod}
\bigl(\tilde \E, D_{\tilde \E} \bigr)\quad \in \Psi_0(\C, C^*(G_1 \times G_2)),
\end{equation}
where $\tilde \E := \tilde \pi(p) \widetilde{\HH_1 \hat \otimes \HH_2}$. Here $p:= p_1 \otimes p_2$, with
$p_j$ the projection in $C_c(X_j \times G_j)$ as defined in \eqref{eq p}. Furthermore, the operator $D_{\tilde \E}$ is the closure of the operator $\tilde D_{\HH_1 \hat \otimes \HH_1}$, as defined in \eqref{eq tilde D}, with $D = D_{\HH_1 \hat \otimes \HH_2}$.

First, let us show that $\tilde \E = \tilde \E_1 \hat \otimes \tilde \E_2$. Note that $\widetilde{\HH_1 \hat \otimes \HH_2}$ is the completion of the space $C_c(G_1 \times G_2, \HH_1 \hat \otimes \HH_2)$ with respect to the $C^*(G_1 \times G_2)$-valued inner product $(\relbar, \relbar)_{\widetilde{\HH_1 \hat \otimes \HH_2}}$, defined analogously to \eqref{eq inprod H tilde}. On the other hand,
\[
\tilde \E_1 \hat \otimes \tilde \E_2 = \tilde \pi_1(p_1) \tilde \HH_1 \hat \otimes \tilde \pi_2(p_2) \tilde \HH_2 = \tilde \pi(p) \tilde \HH_1 \hat \otimes \tilde \HH_2,
\]
since it is not hard to check that $\tilde \pi(f_1 \otimes f_2) = \tilde \pi_1(f_1) \otimes \tilde \pi_2(f_2)$ for all $f_j \in C_c(X_j \times G_j)$. Here $\tilde \HH_1 \hat \otimes \tilde \HH_2$ is the completion of $C_c(G_1, \HH_1) \otimes C_c(G_2, \HH_2)$ in the $C^*(G_1) \otimes C^*(G_2) \cong C^*(G_1 \times G_2)$-valued inner product given by
\[
\bigl(\varphi_1 \otimes \varphi_2, \psi_1 \otimes \psi_2 \bigr)_{\tilde \HH_1 \hat \otimes \tilde \HH_2}
	= (\varphi_1, \psi_1)_{\tilde \HH_1} \otimes (\varphi_2, \psi_2)_{\tilde \HH_2},
\]
for $\varphi_j, \psi_j \in C_c(G_j, \HH_j)$.  It follows directly from the definition \eqref{eq inprod H tilde} of the inner products $(\relbar, \relbar)_{\widetilde{\HH_1 \hat \otimes \HH_2}}$ and $(\relbar, \relbar)_{\tilde \HH_1 \hat \otimes \tilde \HH_2}$, that they coincide on the subspace $C_c(G_1, \HH_1) \otimes C_c(G_2, \HH_2) \subset C_c(G_1 \times G_2, \HH_1 \hat \otimes \HH_2)$.

We claim that the completion of $C_c(G_1, \HH_1) \otimes C_c(G_2, \HH_2)$ with respect to this inner product contains the space $C_c(G_1 \times G_2, \HH_1 \hat \otimes \HH_2)$. Then we  indeed have $\widetilde{\HH_1 \hat \otimes \HH_2} \cong \tilde \HH_1 \hat \otimes \tilde \HH_2$, and hence
\[
 \tilde \E = \tilde \pi(p) \bigl(\widetilde{\HH_1 \hat \otimes \HH_2} \bigr) \cong
	 \tilde \pi(p) \bigl(\tilde \HH_1 \hat \otimes \tilde \HH_2\bigr) = \tilde \E_1 \hat \otimes \tilde \E_2,
\]
as Hilbert $C^*(G_1 \times G_2)$-modules.
The proof of this claim is based on the inequality
\begin{equation} \label{eq ineq norm}
\begin{split}
\|(\varphi, \varphi)_{\tilde \HH_1 \hat \otimes \tilde \HH_2}\|_{C^*(G_1 \times G_2)} & \leq
	\|\varphi\|_{L^1(G_1 \times G_2, \HH_1 \hat \otimes \HH_2)}^2 \\
	& := \left(\int_{G_1 \times G_2} \| \varphi(g_1, g_2) \|_{\HH_1 \hat \otimes \HH_2}  dg_1\, dg_2 \right)^2,
\end{split}
\end{equation}
for all $\varphi \in C_c(G_1, \HH_1) \otimes C_c(G_2, \HH_2)$. This inequality is proved in Lemma \ref{lem ineq norm} below.
Because of this estimate, the completion of $C_c(G_1, \HH_1) \otimes C_c(G_2, \HH_2)$ with respect to the inner product $(\relbar, \relbar)_{\tilde \HH_1 \hat \otimes \tilde \HH_2}$ contains the completion of this tensor product in the norm $\|\cdot \|_{L^1(G_1 \times G_2, \HH_1 \hat \otimes \HH_2)}$, which in turn contains $C_c(G_1 \times G_2, \HH_1 \hat \otimes \HH_2)$.

Next, we prove that the two unbounded cycles \eqref{eq prod ass} and \eqref{eq ass prod} define the same class in $\KK$-theory. By Lemma 10 and Corollary 17 from \cite{Kucerovsky}, this follows if we can show that
\begin{align}
 & \dom D_{\tilde \E_1 \hat \otimes \tilde \E_2} \subset \dom D_{\tilde \E}, \quad \text{and} \label{eq dom incl} \\
 & D_{\tilde \E}|_{\dom D_{\tilde \E_1 \hat \otimes \tilde \E_2}} = D_{\tilde \E_1 \hat \otimes \tilde \E_2}. \label{eq res equal}
\end{align}

We first prove \eqref{eq dom incl}. Note that the domain of $D_{\tilde \E_1 \hat \otimes \tilde \E_2}$ is the completion of $\dom D_{\tilde \E_1} \otimes \dom D_{\tilde \E_2}$ in the norm $\|\cdot \|_{D_{\tilde \E_1 \hat \otimes \tilde \E_2}}$, given by
\begin{equation} \label{eq norm}
\| \varphi_1 \otimes \varphi_2 \|_{D_{\tilde  \E_1 \hat \otimes \tilde \E_2}}^2 :=
	\|\varphi_1 \otimes  \varphi_2\|_{\tilde \HH_1 \hat \otimes \tilde \HH_2}^2 +  \\
		\| D_{\tilde \E_1}  \varphi_1 \otimes \varphi_2 + \varphi_1 \otimes D_{\tilde \E_2}  \varphi_2 \|_{\tilde \HH_1 \hat \otimes \tilde \HH_2}^2,
\end{equation}
for all $\varphi_j \in \dom D_{\tilde \E_j}$.
The domain of $D_{\tilde \E_j}$ in turn is the completion of $\tilde \pi_j(p_j) C_c(G_j, \dom D_j)$ in the norm $\|\cdot\|_{D_{\tilde \E_j}}$, defined analogously to \eqref{eq norm}.

To prove \eqref{eq dom incl}, we consider the subspace
\[
V :=  \tilde \pi_1(p_1) C_c(G_1, \dom D_1) \otimes \tilde \pi_2(p_2) C_c(G_2, \dom D_2)
\]
of $\dom D_{\tilde \E_1} \otimes \dom D_{\tilde \E_2}$.
We begin by showing that the completion of $V$ in the norm $\|\cdot \|_{D_{\tilde \E_1 \hat \otimes \tilde \E_2}}$ contains $\dom D_{\tilde \E_1} \otimes \dom D_{\tilde \E_2}$. This will imply that
\begin{equation} \label{eq V afsl}
\begin{split}
\overline{V} &= \overline{\dom D_{\tilde \E_1} \otimes \dom D_{\tilde \E_2}} \\
	&= \dom D_{\tilde \E_1 \hat \otimes \tilde \E_2},
\end{split}
\end{equation}
with completions taken in the norm $\|\cdot \|_{D_{\tilde \E_1 \hat \otimes \tilde \E_2}}$.

For $j = 1,2$, let $\varphi_j \in \dom D_{\tilde \E_j}$ be given. Let $\bigl(\varphi_j^k \bigr)_{k=1}^{\infty}$ be a sequence in $\tilde \pi_j(p_j) C_c(G_j, \dom D_j)$ such that
\[
\lim_{k \to \infty} \|\varphi_j^k - \varphi_j\|_{D_{\tilde \E_j}} = 0.
\]
We claim that
\begin{equation} \label{eq limiet}
\lim_{k \to \infty} \bigl\| \varphi_1^k \otimes \varphi_2^k - \varphi_1 \otimes \varphi_2 \bigr\|_{D_{\tilde \E_1 \hat \otimes \tilde \E_2}} = 0,
\end{equation}
which implies that $\varphi_1 \otimes \varphi_2$ lies in the completion of $V$ in the norm $\|\cdot \|_{D_{\tilde \E_1 \hat \otimes \tilde \E_2}}$.
This claim is proved in Lemma \ref{lem limiet} below. General elements of $\dom D_{\tilde \E_1} \otimes \dom D_{\tilde \E_2}$ are (finite) sums of simple tensors like $\varphi_1 \otimes \varphi_2$, and can be approximated by sums of sequences like $\bigl( \varphi_1^k \otimes \varphi_2^k \bigr)_{k=1}^{\infty}$. Hence the completion of $V$ in the norm $\|\cdot \|_{D_{\tilde \E_1 \hat \otimes \tilde \E_2}}$ indeed contains $\dom D_{\tilde \E_1} \otimes \dom D_{\tilde \E_2}$, so that \eqref{eq V afsl} holds.

Finally, observe that $\dom D_{\tilde \E}$ is the completion of $\pi(p) C_c(G_1 \times G_2, \dom D_{\HH_1 \hat \otimes \HH_2})$ in the norm $\|\cdot\|_{D_{\tilde \E}}$, which is again defined analogously to \eqref{eq norm}. Since $V$ is contained in $\pi(p) C_c(G_1 \times G_2, \dom D_{\HH_1 \hat \otimes \HH_2})$, the completion of $V$ in the norm $\|\cdot\|_{D_{\tilde \E}}$ is contained in $\dom D_{\tilde \E}$. Furthermore, the operators $D_{\tilde \E}$ and $D_{\tilde \E_1 \hat \otimes \tilde \E_2}$ coincide on $V$, since their restrictions to $V$ are both  given by
\[
\tilde \pi_1(p_1) \varphi_1 \otimes \tilde \pi_2(p_2) \varphi_2 \mapsto
 	\tilde \pi_1(p_1) D_1 \circ \varphi_1 \otimes \tilde \pi_2(p_2) \varphi_2
		+ \tilde \pi_1(p_1) \varphi_1 \otimes \tilde \pi_2(p_2) D_2 \circ \varphi_2.
\]
This implies that the norms $\|\cdot\|_{D_{\tilde \E}}$ and $\|\cdot \|_{D_{\tilde  \E_1 \hat \otimes \tilde \E_2}}$ are the same on $V$, so that the completion of $V$ with respect to $\|\cdot\|_{D_{\tilde \E}}$ equals the completion of $V$ with respect to $\|\cdot \|_{D_{\tilde  \E_1 \hat \otimes \tilde \E_2}}$, which equals $\dom D_{\tilde \E_1 \hat \otimes \tilde \E_2}$, by \eqref{eq V afsl}.
We conclude that
\[
\dom D_{\tilde \E_1 \hat \otimes \tilde \E_2} =\overline{V} \subset \dom D_{\tilde \E},
\]
as claimed.

Claim \eqref{eq res equal} now follows, because by \eqref{eq V afsl}, the restriction of $D_{\tilde \E}$ to $\dom D_{\tilde \E_1 \hat \otimes \tilde \E_2}$ is the closure of $D_{\tilde \E}|_V$, which equals $D_{\tilde \E_1 \hat \otimes \tilde \E_2}|_V$. The closure of the latter operator is $D_{\tilde \E_1 \hat \otimes \tilde \E_2}$, again by \eqref{eq V afsl}, and we are done.
\hfill $\square$

\begin{lemma} \label{lem ineq norm}
The inequality \eqref{eq ineq norm} holds for all $\varphi \in C_c(G_1, \HH_1) \otimes C_c(G_2, \HH_2)$.
\end{lemma}
\begin{proof}
For such $\varphi$, we have
\[
\begin{split}
\|(\varphi, & \varphi)_{\tilde \HH_1 \hat \otimes \tilde \HH_2}\|_{C^*(G_1 \times G_2)}  \leq \|(\varphi, \varphi)_{\tilde \HH_1 \hat \otimes \tilde \HH_2}\|_{L^1(G_1 \times G_2)} \\
& = \int_{G_1 \times G_2} \left|\int_{G_1 \times G_2} \bigl(\varphi(g_1', g_2'), \varphi(g_1'g_1, g_2'g_2) \bigr)_{\HH_1 \hat \otimes \HH_2} dg_1'\, dg_2'\right| \, dg_1\, dg_2 \\
& \leq \int_{G_1 \times G_2} \int_{G_1 \times G_2} \left|\bigl(\varphi(g_1', g_2'), \varphi(g_1'g_1, g_2'g_2) \bigr)_{\HH_1 \hat \otimes \HH_2}\right| dg_1'\, dg_2' \, dg_1\, dg_2 \\
& \leq \int_{G_1 \times G_2} \int_{G_1 \times G_2} \|\varphi(g_1', g_2') \|_{\HH_1 \hat \otimes \HH_2}\,  \| \varphi(g_1'g_1, g_2'g_2) \|_{\HH_1 \hat \otimes \HH_2} dg_1'\, dg_2' \, dg_1\, dg_2,
\end{split}
\]
by the Cauchy-Schwartz inequality.  Because of left invariance of the Haar measures $dg_1$ and $dg_2$, the latter expression is the square of the $L^1$-norm of $\varphi$.
\end{proof}

\begin{lemma} \label{lem limiet}
The limit \eqref{eq limiet} equals zero.
\end{lemma}
\begin{proof}
Since for $j=1,2$, we have
\begin{align}
0 &= \lim_{k \to \infty} \|\varphi_j^k - \varphi_j\|_{D_{\tilde \E_j}}^2 \nonumber \\
	&= \lim_{k \to \infty} \left( \|\varphi_j^k - \varphi_j\|_{\tilde \HH_j}^2 +  \|D_{\tilde \E_j} \varphi_j^k - D_{\tilde \E_j}\varphi_j\|_{\tilde \HH_j}^2 \right), \label{eq normen}
\end{align}
both terms in \eqref{eq normen} tend to zero as $k \to \infty$. Let us rewrite \eqref{eq limiet} in a way that allows us to use this fact. By definition of the norm $\| \cdot \|_{D_{\tilde \E_1 \hat \otimes \tilde \E_2}}$, we have
\begin{multline*}
\bigl\|  \varphi_1^k \otimes \varphi_2^k - \varphi_1 \otimes \varphi_2 \bigr\|_{D_{\tilde \E_1 \hat \otimes \tilde \E_2}}^2 = \\
\bigl\| \varphi_1^k \otimes \varphi_2^k - \varphi_1 \otimes \varphi_2 \bigr\|_{\tilde \HH_1 \hat \otimes \tilde \HH_2}^2 + \\
\bigl\| D_{\tilde \E_1} \varphi_1^k \otimes \varphi_2^k - D_{\tilde \E_1} \varphi_1 \otimes \varphi_2 +
 \tilde  \varphi_1^k \otimes D_{\tilde \E_2} \varphi_2^k
 - \varphi_1 \otimes D_{\tilde \E_2} \varphi_2 \bigr\|_{\tilde \HH_1 \hat \otimes \tilde \HH_2}^2.
\end{multline*}
Using the triangle inequality and the fact that $\|\psi_1 \otimes \psi_2\|_{\tilde \HH_1 \hat \otimes \tilde \HH_2} \leq \|\psi_1\|_{\tilde \HH_1}  \|\psi_1\|_{\tilde \HH_1}$ for all $\psi_j \in \tilde \HH_j$ (this follows from the fact that any $C^*$-norm on a tensor product is \emph{subcross}, see \cite{WO}, Corollary T.6.2), we see that this number is less than or equal to
\begin{multline} \label{eq veel normen}
\Bigl( \| \varphi_1^k - \varphi_1\|_{\tilde \HH_1} \|\varphi_2^k\|_{\tilde \HH_2} +
	\|\varphi_{1}\|_{\tilde \HH_1} \|\varphi_2^k - \varphi_2\|_{\tilde \HH_2} \Bigr)^2 + \\
\Bigl( \| D_{\tilde \E_1} \varphi_1^k - D_{\tilde \E_1} \varphi_1 \|_{\tilde \HH_1} \| \varphi_2^k \|
		+ \|D_{\tilde \E_1} \varphi_1\|_{\tilde \HH_1} \|\varphi_2^k - \varphi_2\|_{\tilde \HH_2} + \Bigr. \\
\Bigl. \| \varphi_1^k - \varphi_1\|_{\tilde \HH_1} \| D_{\tilde \E_2} \varphi_2^k\|_{\tilde \HH_2}	 +
	\|\varphi_{1}\|_{\tilde \HH_1} \| D_{\tilde \E_2} \varphi_2^k  - D_{\tilde \E_2} \varphi_2\|_{\tilde \HH_2}
 \Bigr)^2.
\end{multline}
By the observation at the beginning of this proof, all terms in \eqref{eq veel normen} contain a factor that goes to zero as $k \to \infty$. Since the other factors are bounded functions of $k$, the claim follows.
\end{proof}

\subsection[The induction map $\DInd_K^G$]{The right-hand diagram: a decomposition of the induction map $\DInd_K^G$} \label{right}

In this subsection, we complete the proof of Theorem \ref{natKG} by proving commutativity of the right-hand diagram in \eqref{diagram}.
In this proof, we will use commutativity of the top, middle and bottom diagrams in the case where $N$ is a point.

But first, we give the following description of the map $\DInd_K^G$.
Let $V$ be a finite-dimensional unitary representation of $K$, and let ${\D}^V$ be the Dirac operator defined in \eqref{DV}. The closure of this operator is an unbounded self-adjoint operator on the space of $L^2$-sections of $E_V$, which is odd with respect to the $\Z_2$-grading. This space of $L^2$-sections is isomorphic to the space $\bigl(L^2(G) \otimes \Delta_{d_{\p}} \otimes V\bigr)^K$, where the $K$-action is again defined by \eqref{Kact} (with smooth functions replaced by $L^2$-functions, of course). Let $b$ be a normalising function, so that we have the class
\[
\bigl[\bigl(L^2(G) \otimes \Delta_{d_{\p}} \otimes V\bigr)^K, b({\D}^V), \pi_{G/K} \bigr] \in K_{0}^{G}(G/K).
\]
Here $\pi_{G/K}$ denotes the representation of $C_0(G/K)$ on $L^2(G/K, E_V)$ as multiplication operators.
\begin{lemma} \label{indass}
In this situation, we have
\[
\DInd_K^G[V] = \mu_{G/K}^{G}\bigl[\bigl(L^2(G) \otimes \Delta_{d_{\p}} \otimes V\bigr)^K, b({\D}^V), \pi_{G/K} \bigr] \quad \in K_0(C^*_r(G)).
\]
\end{lemma}
\begin{proof}
Write
\[
[\E, F_{\E}] :=   \mu_{G/K}^{G}\bigl[\bigl(L^2(G) \otimes \Delta_{d_{\p}} \otimes V\bigr)^K, b({\D}^V), \pi_{G/K} \bigr].
\]
Since the restriction of $F_{\E}$ to $\bigl(C_c(G) \otimes \Delta_{d_{\p}} \otimes V\bigr)^K$ is the restriction of $b({\D}^V)$ to this space, we only need to prove that
\begin{equation} \label{eq iso E}
\E = \bigl(C^*_r(G) \otimes \Delta_{d_{\p}} \otimes V\bigr)^K
\end{equation}
as Hilbert $C^*_r(G)$-modules.

To prove this equality, we note that for all $f, f' \in (L^2(G))_c$ and all $g \in G$,
\[
(f,f')_{\E}(g) = (f, g \cdot f')_{L^2(G)} = \bigl(f * (f')^*\bigr)(g),
\]
as one easily computes. This implies that the $C_r^*G$-valued inner product on $\E$ is the same as the one on $\bigl(C^*_r(G) \otimes \Delta_{d_{\p}} \otimes V\bigr)^K$.

The $C^*_r(G)$-module structure of $\E$ is given by
\[
\begin{split}
h\cdot (f\otimes \delta \otimes v) & = \int_G h(g) g\cdot (f\otimes \delta \otimes v) dg \\
	&= (h*f) \otimes \delta \otimes v,
\end{split}
\]
for all $h \in C_c(G)$, $f \in L^2(G)$, $\delta \in \Delta_{d_{\p}}$ and $v \in V$. Hence the equality \eqref{eq iso E} includes the $C^*_r(G)$-module structure.
\end{proof}

\emph{Proof of commutativity of the right-hand diagram.} Consider the vector bundles $V$ and $\{0\}$ over a point. Let $0_V: V \to \{0\}$ be the only possible operator between (the spaces of smooth sections of) these bundles. It defines a class $[0_V] = [V \oplus \{0\}, 0_V] \in K_0^{K}(\mathrm{pt})$, and we have
\[
\mu_{\mathrm{pt}}^K[0_V] = [V] \quad \in R(K).
\]
Now we find that
\[
\DInd_K^G[V] = \mu_{G/K}^{G \times K \times K}\bigl[\bigl(L^2(G) \otimes \Delta_{d_{\p}} \otimes V\bigr)^K, b({\D}^V), \pi_{G/K} \bigr]
\]
by Lemma \ref{indass},
\begin{equation} \label{eq nat diff op}
= \mu_{G/K}^G \circ V_{\Delta(K)} \circ \Res^{G \times K \times K}_{G \times \Delta(K)} [{\D}_{G,K} \otimes 1_V]
\end{equation}
by Corollary 3.13 in \cite{HL} and the fact that ${\D}^V$ is the restriction of ${\D}_{G,K} \otimes 1_V$ to $K$-invariant elements of
$C^{\infty}(G) \otimes \Delta_{d_{\p}} \otimes V$. Corollary 3.13 in \cite{HL} was proved for group actions with a compact orbit space, but the proof given there can easily be generalised to the general case.

By commutativity of the top, middle and bottom diagrams when $N$ is a point, \eqref{eq nat diff op} equals
\[
\begin{split}
&= \mu_{G/K}^G \circ V_{\Delta(K)} \circ \Res^{G \times K \times K}_{G \times \Delta(K)} \bigl( [{\D}_{G,K}] \times [0_V] \bigr) \\
&= R^0_K \circ \Res^{G \times K \times K}_{G \times \Delta(K)} \circ \mu_{G}^{G \times K}\bigl( [{\D}_{G,K}] \times [V] \bigr).
\end{split}
\]
\hfill $\square$

\begin{remark}
Supposing that $V$ is irreducible, we could also have applied the Borel--Weil(--Bott) theorem to realise the class $[V] \in R(K)$ as $\mu_{K/T}^K[{\D}_{i\lambda}]$, where $i\lambda$ is the highest weight of $V$, and ${\D}_{i\lambda}$ is the Dolbeault--Dirac operator on $K/T$ coupled to the usual line bundle that is used in the Borel--Weil theorem. We would then have used commutativity of the top, middle and bottom diagrams for $N=K/T$.
\end{remark}

\section{Dirac operators and the map $\KInd_K^G$} \label{Dirac}

This section is devoted to the proof of Proposition \ref{alphaDirac}. We will define an operator $\tilde {\D}_M^{{L^{2\omega}}}$ whose $K$-homology class is the image of the
class of ${\D}_N^{{L^{2\nu}}}$ under the map $\KInd_K^G$. Then we prove some general facts about principal symbols, and finally we use these facts to
show that ${\D}_M^{{L^{2\omega}}}$ and $\tilde {\D}_M^{{L^{2\omega}}}$ define the same class in $K$-homology, proving Proposition \ref{alphaDirac}.

Throughout this section, we will consider a class
\[
 [N, \nu, \Phi^N, {L^{2\nu}}, (\relbar, \relbar)_{{L^{2\nu}}}, \nabla^N, P^N] \in \CHPS(K),
\]
and we will write
\begin{multline*}
[M, \omega, \Phi^M, {L^{2\omega}}, (\relbar, \relbar)_{{L^{2\omega}}}, \nabla^M, P^M] := \\
\HInd_K^G [N, \nu, \Phi^N, {L^{2\nu}}, (\relbar, \relbar)_{{L^{2\nu}}}, \nabla^N, P^N] \in \CHPS(G).
\end{multline*}

\subsection{Another Dirac operator on $M$} \label{Dtilde}

Let us construct the differential operator $\tilde {\D}_M^{{L^{2\omega}}}$ mentioned in the introduction to this section.
Just like the $\Spin^c$-Dirac operator $\D_M^{{L^{2\omega}}}$, it acts on sections of the spinor bundle
\begin{equation}\label{eq S^M}
\SSS^M := P^M \times_{\Spin^c(d_M)} \Delta_{d_M} \to M,
\end{equation}
associated to the $\Spin^c$-structure $P^M$ defined in Subsection \ref{sec Spinc}.

In the definition of the operator $\tilde {\D}_M^{{L^{2\omega}}}$, we will use the following decomposition of the spinor bundle $\SSS^M$:
\begin{lemma} \label{SpinorM}
We have a $G$-equivariant isomorphism of vector bundles over $M$,
\[
\SSS^M \cong \bigl((G \times \Delta_{d_{\p}}) \boxtimes \SSS^N \bigr)/K,
\]
where $K$ acts on $(G \times \Delta_{d_{\p}}) \boxtimes \SSS^N$ by
\[
k\cdot\bigl( (g, \delta_{\p}) \otimes s^N \bigr) = (gk^{-1}, \widetilde{\Ad}(k)\delta_{\p}) \otimes k\cdot s^N,
\]
for $k \in K$, $g \in G$, $\delta_{\p} \in \Delta_{d_{\p}}$ and $s^N \in \SSS^N$.
\end{lemma}
\begin{proof}
We have the following chain of isomorphisms:
\begin{equation} \label{eq chain iso}
\begin{split}
\SSS^M &\cong \bigl(P_M^{G/K} \times_{\U(1)} (G \times_K P^N) \bigr) \times_H \Delta_{d_{\p}} \otimes \Delta_{d_N} \\
	&\cong \bigl(P^{G/K}_M \times_{\Spin^c(d_{\p})}\Delta_{d_{\p}} \bigr) \otimes \bigl(G \times_K P^N \times_{\Spin^c(d_N)} \Delta_{d_N} \bigr) \\
	&\cong (G \times N \times \Delta_{d_{\p}})/K \otimes (G \times \SSS^N)/K \\
	&\cong \bigl((G \times \Delta_{d_{\p}}) \boxtimes \SSS^N \bigr)/K.
\end{split}
\end{equation}

The first isomorphism in \eqref{eq chain iso} is induced by the $H$-equivariant isomorphism $\Delta_{d_M} \cong \Delta_{d_{\p}} \otimes \Delta_{d_N}$.

The second isomorphism is given by
\[
\bigl[p_M^{G/K}, [g,p^N], \delta_{\p} \otimes \delta_N \bigr] \mapsto [p^{G/K}_M, \delta_{\p}] \otimes \bigl[ [g, p^N], \delta_N \bigr],
\]
for all $p_M^{G/K} \in P_M^{G/K}$, $g \in G$, $p^N \in P^N$, $\delta_{\p} \in \Delta_{d_{\p}}$ and $\delta_{N} \in \Delta_{d_N}$.

The third isomorphism is the obvious one, given the definitions of $P^{G/K}_M$ and $\SSS^N$.

Finally, the fourth isomorphism is a special case of the isomorphism
\[
E/G \otimes F/G \cong (E \otimes F)/G,
\]
if $H$ is a group acting freely  on a manifold $M$, and $E \to M$ and $F \to M$ are $G$-vector bundles.

Explicitly, the isomorphism \eqref{eq chain iso} is given by
\[
\bigl[[g,n,a], [g,p^N], \delta_{d_{\p}} \otimes \delta_N \bigr] \mapsto \bigl[ (g, a\delta_{\p}) \otimes [p^N, \delta_N] \bigr],
\]
for $g \in G$, $n \in N$, $a \in \Spin^c(\p)$, $p^N \in P^N$, $\delta_{\p} \in \Delta_{d_{\p}}$ and $\delta_{N} \in \Delta_{d_N}$.
\end{proof}

Next, let ${\D}_{G,K}$ be the operator defined on page \pageref{DGK}, and consider the operator
\begin{multline*}
{\D}_{G,K} \otimes 1 + 1 \otimes {\D}_N^{{L^{2\nu}}}: \Gamma^{\infty}\bigl(G \times N, (G \times \Delta_{d_{\p}}) \boxtimes \SSS^N\bigr) \to \\
	\Gamma^{\infty}\bigl(G \times N, (G \times \Delta_{d_{\p}}) \boxtimes \SSS^N \bigr),
\end{multline*}
which is odd with respect to the grading on the tensor product $(G \times \Delta_{d_{\p}}) \boxtimes \SSS^N $ induced by the gradings on $\Delta_{d_{\p}}$ and $\SSS^N$.
Because the operators ${\D}_{G, K}$ and ${\D}_N^{{L^{2\nu}}}$ are $K$-equivariant, we obtain an operator\index{daaaaaa@$\tilde {\D}_M^{{L^{2\omega}}}$}
\begin{equation} \label{eq tilde Dirac}
\tilde {\D}_M^{{L^{2\omega}}} := ({\D}_{G,K} \otimes 1 + 1 \otimes {\D}_N^L)^K
\end{equation}
on
\[
\begin{split}
\Gamma^{\infty}\bigl(G \times N, (G \times \Delta_{d_{\p}}) \boxtimes \SSS^N \bigr)^K &\cong
	\Gamma^{\infty}\bigl(M, \bigl((G \times \Delta_{d_{\p}}) \boxtimes \SSS^N \otimes \bigr)/K \bigr) \\
	&\cong \Gamma^{\infty}(M, \SSS^M),
\end{split}
\]
by Lemma \ref{SpinorM}.

The importance of the operator $\tilde {\D}_M^{{L^{2\omega}}} $ lies in the following fact:
\begin{lemma} \label{alphaDtilde}
The image of the class $[{\D}_N^{{L^{2\nu}}}] \in K_0^K(N)$ under the map $\KInd_K^G$ is the class of $\tilde {\D}_M^{{L^{2\omega}}}$ in $K_0^G(M)$.
\end{lemma}
\begin{proof}
By Theorem 10.8.7 from \cite{HR},\footnote{This can also be seen in the unbounded picture of $\KK$-theory.} the Kasparov product $[{\D}_{G,K}] \times [{\D}_N^{{L^{2\nu}}}] \in K_0^{G \times K \times K}(G \times N)$ is the class of the operator ${\D}_{G,K} \otimes 1 + 1 \otimes {\D}_N^{{L^{2\nu}}}$ on $(G \times \Delta_{d_{\p}}) \boxtimes \SSS^N  $. It then follows from Corollary 3.13 in \cite{HL} that the latter class is mapped to the class of $\tilde {\D}_M^{{L^{2\omega}}}$.
\end{proof}
Therefore, Proposition \ref{alphaDirac} follows if we can prove that $\tilde {\D}_M^{{L^{2\omega}}} $ and ${\D}_M^{{L^{2\omega}}}$ define the same $K$-homology class. We prove this fact by showing that their principal symbols are equal.

\subsection{Principal symbols} \label{sec princsymb}

This subsection contains some general facts about the principal symbols of differential operators that are constructed from other differential operators. Their proofs are straightforward.

\subsubsection*{Tensor products}
First, let $X$ and $Y$ be smooth manifolds, and let $E \to X$ and $F \to Y$ be vector bundles.
Let $D_E: \Gamma^{\infty}(E) \to \Gamma^{\infty}(E)$ and
$D_F: \Gamma^{\infty}(F) \to \Gamma^{\infty}(F)$ be differential operators of the same order $d$.
Consider the exterior tensor product $E \boxtimes F \to X \times Y$, and let $D := D_E \otimes 1 + 1 \otimes D_F$ be the operator on $\Gamma^{\infty}(E \boxtimes F)$ given by
\[
D(s\boxtimes t) = D_Es\boxtimes t + s \boxtimes D_Ft,
\]
for $s \in \Gamma^{\infty}(E)$ and $t \in \Gamma^{\infty}(F)$.

As before, we denote the cotangent bundle projection of a manifold $M$ by $\pi_M$. The principal symbols of the operators $D_E$, $D_F$ and $D$ are vector bundle homomorphisms
\[
\begin{split}
\sigma_{D_E}&:\pi_X^*E \to \pi_X^*E; \\
\sigma_{D_F}&:\pi_Y^*F \to \pi_Y^*F; \\
\sigma_{D}&:\pi_{X\times Y}^*(E \boxtimes F) \to \pi_{X\times Y}^*(E \boxtimes F).
\end{split}
\]
Let
\[
\theta:  \pi_{X\times Y}^*(E \boxtimes F) \to \pi_X^*E \boxtimes \pi_Y^*F
\]
be the isomorphism of vector bundles over $T^*(X \times Y) \cong T^*X \times T^*Y$ given by
\[
\theta\bigl((\xi, \eta) , (e\otimes f)\bigr) =  (\xi,e) \otimes (\eta, f),
\]
for $x \in X$, $y \in Y$, $\xi \in T^*_xX$, $\eta \in T^*_yY$, $e \in E_x$ and $f \in F_y$.
The first fact about principal symbols that we will use is:
\begin{lemma} \label{tensor} \index{principal symbol!of a tensor product operator}
The following diagram commutes:
\[
\xymatrix{
\pi_{X\times Y}^*(E \boxtimes F) \ar[rr]^{\sigma_{D}} \ar[d]_{\theta}^{\cong} & & \pi_{X\times Y}^*(E \boxtimes F) \ar[d]_{\theta}^{\cong}\\
\pi_X^*E \boxtimes \pi_Y^*F \ar[rr]_{\sigma_{D_E} \otimes 1 + 1 \otimes \sigma_{D_F}} & & \pi_X^*E \boxtimes \pi_Y^*F.
}
\]
\end{lemma}

\subsubsection*{Pullbacks}

Next, let $X$ and $Y$ again be smooth manifolds, and let $q: E \to Y$ be a vector bundle.
Let $f: X \to Y$ be a smooth map. (We will later apply this to the situation $X = G \times N$, $Y = M$, $E = \SSS^M \otimes {L^{2\omega}}$, and $f$ the quotient map.)
Let $D_E$ be a differential operator on $E$, of order $d$.
Let $D_{f^*E}$ be a differential operator on the pullback bundle $f^*E$ with the property that for all $s \in \Gamma^{\infty}(E)$,
\[
D_{f^*E}(f^*s) = f^*(D_E s).
\]

Consider the vector bundle
\[
f^*(T^*Y \oplus E) \to X.
\]
It consists of triples $(x,\xi, e) \in X \times T^*Y \times E$, with $f(x) = \pi_Y(\xi) = q(e)$. Using this vector bundle, we write down the diagram
\begin{equation} \label{eq diag pullback}
\xymatrix{
\pi_Y^*E \ar[r]^{\sigma_{D_E}} & \pi_Y^*E\\
f^*(T^*Y \oplus E) \ar[u]^a \ar[d]_{b} \ar[r]^{\widetilde{\sigma_{D_E}}} & f^*(T^*Y \oplus E) \ar[u]^a \ar[d]_{b}\\
\pi_X^*(f^*E) \ar[r]^{\sigma_{D_{f^*E}}}& \pi_X^*(f^*E),
}
\end{equation}
where for all $(x, \xi, e) \in f^*(T^*Y \oplus E)$,
\[
\begin{split}
a(x, \xi, e) &:= (\xi, e) \\
b(x, \xi, e) &:= \bigl((T_xf)^*\xi,x, e\bigr) \\
\widetilde{\sigma_{D_E}}(x, \xi, e) &:= \bigl(x, \sigma_{D_E}(\xi, e)\bigr).
\end{split}
\]
\begin{lemma} \label{lem pullback} \index{principal symbol!of a pullback operator}
Diagram \eqref{eq diag pullback} commutes.
\end{lemma}

Rather than diagram \eqref{eq diag pullback}, we would prefer a diagram with a direct vector bundle homomorphism from $\pi_Y^*E$ to $\pi_X^*(f^*E)$
in it. It is however impossible to define such a map in general. The best we can do is to define it for each point $x \in X$ separately: let
\[
(b \circ a^{-1})_x: \pi_Y^*E|_{T^*_{f(x)}Y} \to \pi_X(f^*E)|_{T_x^*X}
\]
be the map
\[
(b \circ a^{-1})_x(\xi, e) = \bigl((T_xf)^*\xi, e\bigr).
\]
Using this map, we obtain the following statement, which is actually equivalent to Lemma \ref{lem pullback}.
\begin{corollary} \label{cor pullback}
For all $x \in X$, the following diagram commutes:
\[
\xymatrix{
\pi_Y^*E|_{T^*_{f(x)}Y} \ar[rr]^{\sigma_{D_E}|_{T^*_{f(x)}Y}} \ar[d]^{(b \circ a^{-1})_x} & & \pi_Y^*E|_{T^*_{f(x)}Y} \ar[d]^{(b \circ a^{-1})_x}\\
\pi_X^*(f^*E)|_{T_x^*X} \ar[rr]^{\sigma_{D_{f^*E}}|_{T_x^*X}} & & \pi_X^*(f^*E)|_{T_x^*X}.
}
\]
\end{corollary}

One last remark that we will use later, is that the maps $(b \circ a^{-1})_x$ are injective if $T_xf$ is surjective. So if $f$ is a submersion, all $(b \circ a^{-1})_x$ are injective.

\subsection{The principal symbols of ${\D}_M^{{L^{2\omega}}}$ and $\tilde {\D}_M^{{L^{2\omega}}}$.} \label{symbDirac}

Let $g^N$ and $g^M$ be the Riemannian metrics on $N$ and $M$, respectively, induced by the $\Spin^c$-structures $P^N$ and $P^M$.
 We use the same notation for the map
$g^M : TM \to T^*M$ given by $v \mapsto g^M(v, \relbar)$, and similarly for $g^N$.
The Dirac operators ${\D}_M^{{L^{2\omega}}}$ and ${\D}_N^{{L^{2\nu}}}$ have principal symbols
\[
\begin{split}
\sigma_{{\D}_M^{{L^{2\omega}}}}: & \pi_M^*\SSS^M \to \pi_M^*\SSS^M;\\
\sigma_{{\D}_N^{{L^{2\nu}}}}: & \pi_N^*\SSS^N \to \pi_N^*\SSS^N,
\end{split}
\]
given by the Clifford action:
\begin{align}
\sigma_{{\D}_M^{{L^{2\omega}}}}(\xi, s^M)   &= \bigl(\xi, c_{TM}\bigl(i (g^M)^{-1}(\xi)\bigr)s^M \bigr); \label{eq sym DMLM} \\
\sigma_{{\D}_N^{{L^{2\nu}}}}(\eta, s^N)   &= \bigl(\eta, c_{TN}\bigl(i(g^N)^{-1}(\eta)\bigr)s^N \bigr), \nonumber
\end{align}
for $m \in M$,  $\xi \in T_m^*M$, $s^M \in \SSS^M_m$ and $n \in N$, $\eta \in T_n^*N$, $s^N \in \SSS^N_n$.

To determine the principal symbol of $\tilde {\D}_M^{{L^{2\omega}}}$, we need the following basic fact:
\begin{lemma} \label{lem symb DGK} \index{principal symbol!of the operator ${\D}_{G,K}$}
The principal symbol of the operator ${\D}_{G,K}$ on the trivial bundle $G \times \Delta_{d_{\p}} \to G$ is given by
\[
\sigma_{{\D}_{G,K}}(g, \xi, \delta_{\p}) = (g, \xi, c_{\p}(i\xi_{\p^*}) \delta_{\p}),
\]
for $g \in G$, $\xi \in \g^*$ and $\delta_{\p} \in \Delta_{d_{\p}}$. Here $\xi_{\p^*}$ is the component of $\xi$ in $\p^* \cong \kk^0$ according to $\g^* = \p^0 \oplus \kk^0$, and we identify $\p^*$ with $\p$, and $\p$ with $\R^{d_{\p}}$, using a $B$-orthonormal basis $\{X_1, \ldots, X_{d_{\p}}\}$ of $\p$.
\end{lemma}
\begin{proof}
Let $g \in G$, $f \in C^{\infty}(G)$ and $\tau \in C^{\infty}(G, \Delta_{d_{\p}})$ be given. Then
\[
\begin{split}
\sigma_{{\D}_{G,K}}(d_gf, \tau(g)) &= \bigl(d_gf, \lim_{\lambda \to  \infty} \frac{1}{\lambda} \bigl(e^{-i\lambda f} {\D}_{G,K} (e^{i\lambda f}\tau) \bigr)(g)\bigr) \\
	&=
\bigl(d_gf, \lim_{\lambda \to  \infty} \frac{1}{\lambda} \bigl(e^{-i\lambda f} \sum_{j} c_{\p}(X_j) X_j(e^{i\lambda f}\tau) \bigr)(g)\bigr). \\
\end{split}
\]
This expression equals
\begin{multline*}
\Bigl(d_gf, \lim_{\lambda \to  \infty} \frac{1}{\lambda} \bigl( \sum_{j} c_{\p}(X_j)\bigl(i\lambda X_j(f)\tau + X_j(\tau) \bigr) \bigr)(g)\Bigr)
	\\ = \Bigl(d_gf, i\sum_j c_{\p}(X_j) \langle d_gf, T_el_g (X_j) \rangle \tau(g) \Bigr).
\end{multline*}

Hence for all $ \xi  \in \g^*$, $\delta_{\p} \in \Delta_{d_{\p}}$, we have
\[
\begin{split}
\sigma_{{\D}_{G,K}}(g, \xi, \delta_{\p}) &= \Bigl(g, \xi, i\sum_j c_{\p}(\langle \xi, X_j\rangle X_j) \delta_{\p} \Bigr) \\
	&=  \left(g, \xi,  c_{\p}( i\xi_{\p}) \delta_{\p} \right),
\end{split}
\]
since $\{X_j\}$ is a basis of $\p$, orthonormal  with respect to the Killing form.
\end{proof}

We are now ready to prove that $\D_M^{{L^{2\omega}}}$ and $\tilde {\D}_M^{{L^{2\omega}}}$ have the same principal symbol, and hence define the same class in $K$-homology. This will conclude the proof of Proposition \ref{alphaDirac}, which was the remaining step in the proof of Theorem \ref{thm [Q,I]=0}. As we saw in Subsection \ref{sec [Q,I] = [Q,R]}, the latter theorem implies Theorem \ref{GSss}, which is our second main result.
\begin{proposition} \label{prop tilde sigma}
The following diagram commutes:
\begin{equation} \label{diagsigmaDtilde}
\xymatrix{
\pi_M^* \SSS^M \ar[r]^{\sigma_{{\D}_M^{{L^{2\omega}}}}} \ar[d]^{\cong} & \pi_M^*\SSS^M  \ar[d]^{\cong} \\
\pi_M^*\bigl( \bigl((G \times \Delta_{d_{\p}}) \boxtimes \SSS^N \bigr)/K \bigr) \ar[r]^{\sigma_{\tilde {\D}_M^{{L^{2\omega}}}}}  & \pi_M^*\bigl( \bigl((G \times \Delta_{d_{\p}}) \boxtimes \SSS^N \bigr)/K \bigr)  \\
p^* \bigl(T^*M \oplus ((G \times \Delta_{d_{\p}}) \boxtimes \SSS^N )/K \bigr) \ar[u]^{a} \ar[d]_{b} \ar[r]^{\widetilde{\sigma_{\tilde {\D}_M^{{L^{2\omega}}}}}} & p^* \bigl(T^*M \oplus ((G \times \Delta_{d_{\p}}) \boxtimes \SSS^N )/K \bigr) \ar[u]^{a} \ar[d]_{b}\\
\pi_{G \times N}^* \bigl(p^*  ((G \times \Delta_{d_{\p}}) \boxtimes \SSS^N )/K\bigr) \ar[d]^{\cong}_{h} \ar[r]&
		\pi_{G \times N}^* \bigl(p^*  ((G \times \Delta_{d_{\p}}) \boxtimes \SSS^N )/K\bigr) \ar[d]^{\cong}_{h} \\
\pi_{G \times N}^* \bigl((G \times \Delta_{d_{\p}}) \boxtimes \SSS^N \bigr) \ar[d]^{\cong}_{\theta} \ar[r]^{\sigma_{{\D}_{G,K} \otimes 1 + 1 \otimes {\D}_N^{{L^{2\nu}}}}} &
\pi_{G \times N}^* \bigl((G \times \Delta_{d_{\p}}) \boxtimes \SSS^N \bigr) \ar[d]^{\cong}_{\theta} \\
\pi_G^*(G \times \Delta_{d_{\p}}) \boxtimes \pi_N^*\SSS^N  \ar[r]^{\sigma_{{\D}_{G,K}} \otimes 1 + 1 \otimes \sigma_{{\D}_N^{L^{2\nu}}}} &
	\pi_G^*(G \times \Delta_{d_{\p}}) \boxtimes \pi_N^*\SSS^N .
}
\end{equation}
Here the isomorphism $h$ is induced by the general isomorphism $p^*(E/H) \cong E$.
The fourth horizontal map from the top is just defined as the composition $h^{-1}\circ (\sigma_{{\D}_{G,K} \otimes 1 + 1 \otimes {\D}_N^{{L^{2\nu}}}})\circ h$, i.e.\ by commutativity
of the second square from the bottom.
\end{proposition}
\begin{proof}
It follows from Lemma \ref{tensor} that the bottom square of \eqref{diagsigmaDtilde} commutes. Note that
\[
\bigl({\D}_{G,K} \otimes 1 + 1 \otimes {\D}_N^{{L^{2\nu}}} \bigr)p^*s = p^* \bigl(\tilde {\D}_M^{{L^{2\omega}}}s \bigr)
\]
for all $s \in \Gamma^{\infty}\bigl( \bigl((G \times \Delta_{d_{\p}}) \boxtimes \SSS^N  \bigr)/K\bigr)$.
We can therefore apply Lemma \ref{lem pullback} to see that the second and third squares in \eqref{diagsigmaDtilde} from the top commute as well.
 We will first show that the outside of diagram \eqref{diagsigmaDtilde} commutes, and then deduce commutativity of the top subdiagram.

Let $g \in G$, $n \in N$, $\eta \in T^*_nN$, $\xi \in \p^*$,  $p^N \in P^N$, $\delta_{\p} \in \Delta_{d_{\p}}$ and $\delta_N \in \Delta_{d_N}$ be given. Then we have the element
\begin{equation} \label{eq elt A}
\bigl((g,n), [g, \eta, \xi], \bigl[(g, \delta_{\p}) \otimes [p^N, \delta_N]  \bigr] \bigr) \in
	p^* \bigl(T^*M \oplus ((G \times \Delta_{d_{\p}}) \boxtimes \SSS^N )/K \bigr).
\end{equation}
Here we have used Proposition \ref{TM} and Lemma \ref{TM2}.
Applying the map $a$ and the (inverse of the) isomorphism in the upper left corner of \eqref{diagsigmaDtilde} to this element, we obtain
\begin{multline} \label{eq a van iets}
\bigl([g, \eta, \xi], \bigl[[g,n,e_{\Spin^c(\p)}], [g,p^N], \delta_{\p} \otimes \delta_N \bigr] \bigr) \\
 \in \pi_M^*\bigl(P_M^{G/K} \times_{\U(1)} (G \times_K P^N) \times_H \Delta_{d_{\p}} \otimes \Delta_{d_N} \bigr) \\
\cong \pi_M^*\SSS^M.
\end{multline}
Here $e_{\Spin^c(\p)}$ is the identity element of $\Spin^c(\p)$.

Let $\zeta \in \bigl(\R^{d_N} \bigr)^*$ be the covector such that $\eta \in T^*N$ corresponds to $[p^N, \zeta] \in P^N \times_{\Spin^c(d_N)} \bigl(\R^{d_N} \bigr)^*$. Then $\sigma_{\tilde {\D}_M^{{L^{2\omega}}}}$ applied to \eqref{eq a van iets} gives
\[
\bigl([g, \eta, \xi], \bigl[[g,n,e_{\Spin^c(\p)}], [g,p^N],
	c_{\p \oplus \R^{d_N}}(i\xi, i\zeta) (\delta_{\p} \otimes \delta_N) \bigr] \bigr),
\]
where we identify $\bigl(\R^{d_N} \bigr)^* \cong \R^{d_N}$ using the standard Euclidean metric, and $\p^* \cong \p$ using the Killing form.
By definition of the Clifford modules $\Delta_k$ (see e.g.\ \cite{Friedrich}, page 13), this equals
\[
\bigl([g, \eta, \xi], \bigl[[g,n,e_{\Spin^c(\p)}], [g,p^N],
	c_{\p}(i\xi)\delta_{\p} \otimes \delta_N + \delta_{\p} \otimes c_{\R^{d_N}}(i\zeta) \delta_{d_N} \bigr] \bigr).
\]
(This is the central step in the proof of Proposition \ref{alphaDirac}.)

The image of the latter element under the maps $\theta \circ h \circ (b \circ a^{-1})_{(g,n)}$ is
\[
\bigl((g,\xi), (g, c_{\p}(i\xi)\delta_{\p}) \bigr) \otimes \bigl( \eta, [p^N, \delta_N] \bigr) +
\bigl((g,\xi), (g, \delta_{\p}) \bigr) \otimes \bigl( \eta, [p^N, c_{\R^{d_N}}(i\zeta)\delta_N] \bigr), \
\]
which  by Lemma \ref{lem symb DGK} equals the image under the map
\[
\bigl(\sigma_{{\D}_{G,K}} \otimes 1 + 1 \otimes \sigma_{{\D}_N^{{L^{2\nu}}}} \bigr) \circ \theta \circ  h \circ b
\]
of \eqref{eq elt A}.
Therefore, the outside of diagram \eqref{diagsigmaDtilde} commutes.

Now note that for all $(g,n) \in G \times N$, the composition $\theta \circ h \circ (b \circ a^{-1})_{(g,n)}$ is injective, because $p$ is a submersion (see the remark after Corollary \ref{cor pullback}). This fact, together with commutativity of the outside of diagram \eqref{diagsigmaDtilde}, implies that the top part of \eqref{diagsigmaDtilde} commutes as well.
\end{proof}

\subsection*{Assumptions and notation}

In this paper, we have used the following assumptions and notation.

\subsection*{Assumptions}

\begin{itemize}
\item All manifolds and all maps between them are supposed to be smooth. In particular, all group actions are smooth.
\item All momentum maps are supposed to be equivariant with respect to the coadjoint action.
\item Unless stated otherwise, all vector bundles except those constructed from tangent bundles are supposed to be complex.
\end{itemize}

\subsection*{Notation}

\subsubsection*{Groups}

\begin{itemize}
\item $H$: a group;
\item $G$: connected semisimple Lie group with finite centre (except in Subsections \ref{sec tangent}, \ref{sec top} and \ref{sec bottom});
\item $K<G$: maximal compact subgroup;
\item $T<K$: maximal torus, also supposed to be a Cartan subgroup of $G$ (i.e.\ $\rank G = \rank K$);
\item $\torus \subset \kk \subset \g$: the respective Lie algebras;
	 (we identify the dual space $\torus^*$ with the subspace $(\kk^*)^{\Ad^*(T)}$ of $\Ad^*(T)$-invariant elements);
\item $B$: the Killing form on $\g$;
\item $\p \subset \g$: the ($\Ad(K)$-invariant) orthogonal complement of $\kk$ in $\g$ with respect to $B$;
\item $V^{0}$: for a subspace $V$ of a vector space $W$, the annihilator of $V$ in $W^*$, i.e.\ the space $\{\xi \in W^*; \xi|_V = 0\}$
	 (we identify $\kk^*$ with the annihilator $\p^0 \subset \g^*$, and $\p^*$ with $\kk^0$);
\item $\torus^*_+ \subset \torus^*$: a choice of positive Weyl chamber;
\item $R = R(\g, \torus)$: the set of roots of $(\g, \torus)$;
\item $R_c = R(\kk, \torus)$: the set of roots of $(\kk, \torus)$, considered as a subset of $R$;
\item $R_n = R \setminus R_c$: the set of noncompact roots of $(\g, \torus)$;
\item $R^+$: the set of positive roots of $(\g, \torus)$ with respect to $\torus^*_+$;
\item $R_c^+, R_n^+$:  $R_c \cap R^+$ and $R_n \cap R^+$, respectively;
\item $\rho, \rho_c, \rho_n$: half the sum of the positive roots in $R^+$, $R_c^+$ and $R_n^+$, respectively;
\item $W(\g, \torus), W(\kk, \torus)$: the Weyl groups of $(\g, \torus)$ and $(\kk, \torus)$, respectively.
\item $T^{\reg}$: the dense subset of regular elements of $T$:
	$T^{\mathrm{reg}}:= \{\exp X; X \in \torus, (\alpha, X) \not\in 2\pi i \, \Z \text{ for all $\alpha \in R(\g, \torus)$}\}$;
\item $\ncw \subset \torus^*$: the union of the `noncompact walls', i.e.\ the set of
	 $\xi \in \torus^*$ such that for some $\alpha \in R_n$, we have $(\alpha, \xi) =0$;
\item $\g^*_{\mathrm{ell}}$; the set of elliptic elements of $\g^*$, equal to $\Ad(G)\kk^*$;
\item $\gse$; the set of strongly elliptic elements of $\g^*$, equal to the interior of $\gse$, to the set of elements of $\g^*$ with compact stabilisers, and to
$\Ad(G)(\torus^*_+ \setminus \ncw)$;
\item $\O^{\xi}$: for $\xi \in \g^*$ or $\xi \in \kk^*$, the coadjoint orbit $\Ad^*(G)\xi$ of $G$ or the coadjoint orbit $\Ad^*(K)\xi$ of $K$, where appropriate;
\item $\O^{\lambda}$: for $\lambda \in i\g^*$ or $\lambda \in i\kk^*$, the coadjoint orbit $\O^{-i\lambda}$;
\item $I_k$; for $k \in \N$, the $k\times k$ identity matrix.
\end{itemize}

\subsubsection*{Representations}

\begin{itemize}
\item $R(K)$: the representation ring of $K$;
\item $m V$: for $m \in \Z$ and $V$ a representation space of $K$, the $m$-fold direct sum $V \oplus \cdots \oplus V$ if $m>0$, minus the $|m|$-fold direct sum $V \oplus \cdots \oplus V$ if $m<0$, and the zero space if $m=0$;
\item $\Lambda_+ \subset i\torus^*_+$ or $\Lambda^{\kk}_+$: the set of dominant weights of $(\kk, \torus)$ with respect to $\torus^*_+$;
\item $V_{\lambda}$: for $\lambda \in \Lambda_+$, the irreducible representation of $K$ with highest weight $\lambda$;
\item $\chi_V$: for $V$ a representation, the character of $V$;
\item $\chi_{\lambda}$: for $\lambda \in \Lambda_+$, the character of $V_{\lambda}$;
\item $[V:W]$: for two representations $V, W$ of a group $H$, the multiplicity of $W$ in $V$, equal to
  $\dim \Hom(V, W)^K$;
\item $\Delta_{2k+1}$: for $K \in \N$  the canonical irreducible representation of the group $\Spin(2k+1)$ (see \cite{Friedrich});
\item $\Delta_{2k} = \Delta_{2k}^+ \oplus \Delta_{2k}^-$: for $k \in \N$, the canonical representation of $\Spin(k)$, split into two irreducible subrepresentations;
\item $c:V \to \End(\Delta_V)$, for $V$ a vector space equipped with a bilinear form, the Clifford action of $V$ on $\Delta_V$ (see \cite{Friedrich});
\item $\widetilde{\Ad}$: the homomorphism (if it exists) $K \to \Spin(\p)$ such that $\lambda \circ \widetilde{\Ad} = \Ad: K \to \SO(\p)$, with $\lambda: \Spin(\p) \to \SO(\p)$ the double covering map;
\item $\HH$: a Hilbert space;
\item $R_G^{\HH}$: for an irreducible discrete series representation $\HH$ of $G$, the
  reduction map for $G$ defined in \eqref{RG};
\item $\K(\HH)$: for $\HH$ a Hilbert space, the algebra of compact operators on $\HH$.
\end{itemize}

\subsubsection*{Topological spaces, manifolds and vector bundles}

For any topological space $X$,
\begin{itemize}
\item $C_0(X)$: the space of (complex valued) continuous functions on $X$ that vanish at infinity;
\item $C_c(X)$: the space of (complex valued) compactly supported continuous functions on $X$;
\item $h\cdot s$: for $h$ an element of a group $H$ acting on  $X$, and $s$ a section of an $H$-vector bundle over $X$, the section given by
$(h\cdot s)(x) = h\cdot s(h^{-1}x)$;
\item $E \boxtimes F$: for $E \to X$ a vector bundle and $F \to Y$ a vector bundle over another space, the exterior tensor product $p_X^*E \otimes p_Y^*F \to X \times Y$, with $p_X, p_Y: X \times Y \to X, Y$ the canonical projections.
\end{itemize}

\noindent
For any manifold $M$,
\begin{itemize}
\item $d_M$:  the dimension of $M$;
\item $C^{\infty}(M)$: the space of (complex valued) smooth functions on $M$;
\item $\pi_M$:  the cotangent bundle projection $\pi_M: T^*M \to M$;
\item $\mathfrak{X}(M)$:  the space of vector fields on $M$;
\item $X_M$: for $X$ in the Lie algebra of a group acting on  $M$, the induced vector field on $M$
	(the subscript $M$ will often be omitted);
\item $\Gamma^{\infty}(E)$: for a smooth vector bundle $E$ over a given manifold, the space of smooth sections of $E$, also denoted by $\Gamma^{\infty}(M, E)$;
\item $R_{\nabla}$: for $\nabla$ a connection on a vector bundle, the curvature of $\nabla$;
\item $H_m, \h_m$: for a given action of a Lie group $H$ on  $M$, and a point $m \in M$,
  the global and infinitesimal stabiliser of $m$, respectively;
\item $(P^M, \psi^M)$: a $\Spin^c$-structure on $M$, that is, a principal $\Spin^c(d_M)$-bundle $P^M \to M$ and an isometric vector bundle isomorphism $\psi^M: P^M \times_{\Spin^c(d_M)} \R^{d_M} \to TM$;
\item $\SSS^M$: if $M$ has a $\Spin^c$-structure $(P^M, \psi^M)$, the spinor bundle $\SSS^M = P^M \times_{\Spin^c(d_M)} \Delta_{d_M}$;
\item $\D_M^E$: if $M$ has a  $\Spin^c$-structure and $E \to M$ is a vector bundle with a connection, the $\Spin^c$-Dirac operator on $M$ coupled to $E$ (see \cite{Duistermaat}, \cite{Friedrich});
\item $\D_{G,K}$: the differential operator on the trivial bundle $G \times \Delta_{d_{\p}} \to G$ given by \eqref{DV}, with $V=\C$ the trivial representation.
\end{itemize}

\subsubsection*{Symplectic geometry}

\begin{itemize}
\item $(M, \omega)$: a symplectic manifold carrying a Hamiltonian action of $G$
  (in Subsection \ref{sec ind}, this is to be proved);
\item $P^M$, a $G$-equivariant $\Spin^c$ structure on $M$, whose determinant line bundle has Chern class compatible $2\omega$;
\item $\Phi^M$: the momentum map of this action;
\item $(N, \nu)$: a symplectic manifold carrying a Hamiltonian action of $K$
  (in Subsection \ref{sec res}, this is to be proved);
\item $P^N$:  a $K$-equivariant $\Spin^c$ structure on $N$, whose determinant line bundle has Chern class compatible $2\nu$;
\item $\Phi^N$: the momentum map of this action;
\item $\Phi^M_X, J_Y$: for $X \in \g$ and $Y \in \kk$, the pairings $\langle \Phi^M, X \rangle$ and $\langle \Phi^N, X \rangle$, respectively;
\item $M_{\xi}$: for $\xi \in \g^*$, the symplectic reduction $\bigl(\Phi^M\bigr)^{-1}(\xi)/G_{\xi}$;
\item $N_{\xi}$: for $\xi \in \kk^*$, the symplectic reduction $\bigl(\Phi^N\bigr)^{-1}(\xi)/K_{\xi}$;
\item $M_{\lambda}$, $N_{\lambda}$: for $\lambda \in i\g^*$ or $i\kk^*$ respectively, the symplectic reductions $M_{-i\lambda}$ and $N_{-i\lambda}$.
\end{itemize}

\subsubsection*{(Unbounded) $\KK$-theory}
\begin{itemize}
\item $\Psi^H(A,B)$: for a group $H$, and $A$ and $B$ $H$-$C^*$-algebras, the semigroup of equivariant unbounded Kasparov cycles over $A$ and $B$ (see \cite{BJ}).
\end{itemize}

\noindent
Peter Hochs \\
Institute for Mathematics, Astrophysics and Particle Physics \\
Radboud University \\
Toernooiveld 1 \\
6525 ED Nijmegen \\
The Netherlands\\
\texttt{peter.hochs@tno.nl}


\end{document}